\documentclass[13pt]{amsart}
   \usepackage{amsmath}
   \usepackage{}
   \usepackage{amsfonts}
   \usepackage{amsfonts,latexsym,rawfonts,amsmath,amssymb,amsthm}
   \usepackage[plainpages=false]{hyperref}

   \usepackage{pdfsync}

   \usepackage[a4paper]{geometry}
   \usepackage{esint}
   \usepackage{graphics,graphicx}
   \usepackage{tikz}

   \numberwithin{equation}{section}
   \newcommand{\beq}{\begin{equation}}
   \newcommand{\eeq}{\end{equation}}
   \newcommand{\beqs}{\begin{eqnarray*}}
   \newcommand{\eeqs}{\end{eqnarray*}}
   \newcommand{\beqn}{\begin{eqnarray}}
   \newcommand{\eeqn}{\end{eqnarray}}
   \newcommand{\beqa}{\begin{array}}
   \newcommand{\eeqa}{\end{array}}

   \def\lra{\longrightarrow}

   \def\bc{\begin{center}}
   \def\ec{\end{center}}

   \def\begeq{\begin{equation}}
   \def\endeq{\end{equation}}
   \def\and{\quad{\rm and}\quad}

   \let\lra=\longrightarrow

   \def\mapright\#1{\,\smash{\mathop{\lra}\limits^{\#1}}\,}

   \newtheorem{prop}{Proposition}[section]
   \newtheorem{theo}[prop]{Theorem}
   \newtheorem{lem}[prop]{Lemma}
   
   \newtheorem{cor}[prop]{Corollary}
   \newtheorem{rem}[prop]{Remark}
   
   \newtheorem{defi}[prop]{Definition}

\begin{document}

    \title{ Uniformly strong convergence of K\"ahler-Ricci flows  on a Fano manifold }

   %\date{-}

   \author{Feng Wang{\dag}  and Xiaohua $\text{Zhu}^{*}$}

   \address{School of Mathematical Sciences, Peking
   University, Beijing 100871, China.}

   \email{ xhzhu@math.pku.edu.cn}
\thanks {$\dag$  Partially supported by NSFC Grants 11971423}
   \thanks {* Partially supported by NSFC Grants 11771019 and BJSF Grants Z180004.}

   \begin{abstract}
   In this paper, we study the uniformly strong convergence of K\"ahler-Ricci flow  on a Fano manifold with varied initial metrics  and smooth deformation  complex structures.   As an application,  we prove the uniqueness of K\"ahler-Ricci  solitons in sense of diffeomorphism orbits.   The result   generalizes Tian-Zhu's   theorem  for  the uniqueness of  of K\"ahler-Ricci  solitons on a  compact complex manifold,  and  it  is  also  a generalization of   Chen-Sun's  result of for    the uniqueness of  of K\"ahler-Einstein  metric orbits.
   \end{abstract}

    \date{}

 \maketitle

 \tableofcontents

 \setcounter{section}{-1}

    \section{Introduction}

  Let $(M, J, \omega_0)$ be a Fano manifold and $\omega_t$  a solution of normalized K\"ahler-Ricci flow,
   \begin{align}\label{kr-flow}
   \frac{\partial \omega_t}{\partial t}=-{\rm Ric}\,(\omega_t)+\omega_t,~\omega_0\in 2\pi c_1(M, J).
   \end{align}
   It is known in \cite{cao} that (\ref{kr-flow})  has a global solution $\omega_t$ for all $t\ge 0$ whenever the initial metric  $\omega_0$ represents $2\pi c_1(M)$. Thus the main interesting of  (\ref{kr-flow}) is to study the limit behavior of  $\omega_t$. The famous
   Hamilton-Tian  conjecture (simply written as  HT- conjecture) \footnote{We  also write KR soliton as  K\"ahler-Ricci soliton,  KE metric as K\"ahler-Einstein  metric and others,  below.}
   asserts \cite{Tian97}:

    {\it Any sequence of $(M, \omega_t)$ contains a subsequence converging to a length space $(M_\infty,\omega_\infty)$
   in the Gromov-Hausdorff topology and $(M_\infty,\omega_\infty)$ is a smooth K\"ahler-Ricci soliton outside a closed subset $S$, called the singular set, of codimension at least $4$. Moreover, this subsequence of $(M, \omega_t)$ converges  locally to the regular part of $(M_\infty,\omega_\infty)$
   in the Cheeger-Gromov topology.}

   The Gromov-Hausdorff convergence part in the HT conjecture follows from Perelman's non-collapsing result and Zhang's upper volume estimate \cite{Zhq}. More recently, there were very significant progresses on this conjecture, first by Tian and Zhang in dimension less than $4$ \cite{TZhzh}, then by Chen-Wang \cite{Chwang} and Bamler \cite{Bam} in higher dimensions. In fact, Bamler proved a generalized version of the conjecture.

  Based a work of Liu-Sz\'ekelyhidi on  Tian's  partical $C^0$-estimate for  poralized K\"ahler metrics with Ricci bounded  below \cite{LS},  the authors  recently  gave  an alternative  proof to  the HT-conjecture as follows.
  \begin{theo}\label{WZ} For any sequence of $(M, \omega_t)$ of  (\ref{kr-flow}),
   there is   a Q-Fano variety $\tilde M_\infty$ with klt singularities such that  $\omega_{t_i}$ (after taking a subsequence)  is locally  $C^\infty$-convergent to a  KR soliton  $\hat \omega_\infty$ on  ${\rm Reg}(\tilde M_\infty)$ in the Cheeger-Gromov topology.  Moreover,    $\hat\omega_\infty$ can be extended  to a  singular  KR soliton  on $\tilde M_\infty$ with  a $L^\infty$-K\"ahler potential $\phi_\infty$ and the completion of $({\rm Reg}(\tilde M_\infty), \hat \omega_\infty)$ is isometric to the global limit  $(M_\infty, \omega_\infty)$ of  $\omega_{t_i}$  in  the Gromov-Hausdroff   topology.   In  addition,    if $\hat\omega_\infty$ is a  singular  KE  metrics,  $\phi_\infty$ is continuous and $M_\infty$   is homeomorphic to  $\tilde M_\infty$ which has  at least 4   Hausdroff codimension  of singularities  of  $(M_\infty, \omega_\infty)$.
   \end{theo}

  With respect to HT-conjecture,   it is natural to ask whether  the Q-Fano variety $\tilde M_\infty$ in Theorem \ref{WZ} is unique and independent of initial metric $\omega_0$ or not.  The positive answer will implies that the flow  (\ref{kr-flow})  converges   locally uniformly to a singular KR soliton and the  Gromov-Hausdroff limit is unique, since  the singular KR soliton is unique on a Q-Fano variety  with klt singularities \cite{Bern}.
   In this paper, we  solve the problem  under assumption that  the connected component  ${\rm Aut}_0(\tilde M_\infty)$  containing the identity  of the auto-morphisms group of  $\tilde M_\infty$ is reductive.
  In fact,   we  prove the following uniformly  strong  convergence of  (\ref{kr-flow}).

  \begin{theo}\label{general-KR} Let $(\tilde M_\infty, \hat \omega_\infty)$ be a  singular  KR soliton  which is a limit of  sequence   $\{(M, \omega_{t_i})\}$ of  (\ref{kr-flow}) as in Theorem \ref{WZ}. Suppose that ${\rm Aut}_0(\tilde M_\infty)$ is reductive. Then  for any  initial metric $\omega_0'\in 2\pi c_1(M)$,   (\ref{kr-flow}) is locally uniformly  $C^\infty$-convergent to $({\rm Reg}(\tilde M_\infty), \hat \omega_\infty)$  in the Cheeger-Gromov topology.  As a consequence, the Gromov-Hausdroff limit of (\ref{kr-flow})  is  independent of choice of sequences  and initial metrics in  $2\pi c_1(M)$.
  \end{theo}

 For a fixed flow (\ref{kr-flow}) with an initial metric $\omega_0$,   the uniqueness of algebraic structures of limits $\tilde M_\infty$ in Theorem \ref{WZ}, or  Theorem \ref{general-KR} has been confirmed by  Chen-Sun-Wang \cite{CSW} even without assuming that ${\rm Aut}_0(\tilde M_\infty)$ is reductive ( also see Theorem \ref{mainThm-KS} below).   Their proof is based on Chen-Wang's  work on HT-conjecture to derive the  partial $C^0$-estimate \cite{Chwang} so that  the uniqueness  of $\tilde M_\infty$ is reduced to studying  a finite dimension problem  in sense of Hilbert-Mumford's GIT figure.
 They use  a  method  to  estimate  the eigenvalues and  eigenspaces for representation group in order to prove that   $\tilde M_\infty$ as a Chow point
 lies in an orbit of  reductive subgroup $G_{v}$ of ${\rm SL}(N+1, \mathbb C)$ associated to  another normal  $\overline  M_\infty$,  where $v$ is the soliton VF of $\hat \omega_\infty$   (also see Section 5).  The choice of  $G_{v}$ makes $G_{v}\cap {\rm Aut}_0(\tilde M_\infty)={\rm Aut}_r(\tilde M_\infty)$ so that GIT can be applied, where  ${\rm Aut}_r(\tilde M_\infty) $ is a reductive subgroup of ${\rm Aut}_0(\tilde M_\infty)$  (cf.  Proposition \ref{redsol}, also see  \cite{BW}).

 For  flow   $(M, \omega_t')$ of (\ref{kr-flow})  with  varied  initial metrics,  it seems   that GIT's method in  \cite{CSW} can not be directly  applied because of  lack of uniqueness of  reductive subgroup $G_{v}$.  However, by assuming  that  the  total ${\rm Aut}_0(\tilde M_\infty)$ is  reductive,  we are able to deform  (\ref{kr-flow}) with a simple   path of  initial metrics in $2\pi c_1(M)$ to study the limit behavior of $\omega_t'$ as done for the smooth limits of $\omega_t'$ in \cite{TZ4, TZZZ}. Moreover,  we can  generalize our method to study the convergence of KR flow on a Fano manifold with  jumping complex structure.

 \begin{defi}\label{deformation-j}Let $(M,J)$ be a Fano manifold. A complex manifold  $(M', J')$  is called a canonical smooth deformation of  $(M,J)$
 if there are a sequence of K\"ahler metrics $\omega_i$ in $2\pi c_1(M,J)$
 and  diffeomorphisms $\Psi_i: M'\to M$ such that
 \begin{align}\label{c-infty-convergence-vector}\Psi_i^* \omega_i     \stackrel{C^\infty}{\longrightarrow} \omega', ~ \Psi_i^* J \stackrel{C^\infty}{\longrightarrow} J', ~{\rm on}~ M'.
 \end{align}
 In addition that  $J'$ is not conjugate to $J$,  $J'$  is called a jump of $J$.

 \end{defi}

 We prove

 \begin{theo}\label{stability-KR-smooth-complex} Let $(M', J')$  is  a canonical smooth jump of a Fano manifold  $(M,J)$.  Suppose that  ${\rm Aut}_0( M')$ is reductive and $(M', J')$ admits a KR soliton $\omega_{KR}$ such that
 \begin{align}\label{max-entropy-4}\lambda(\omega_{KR})=\sup\{\lambda(g')|~\omega_{g'}\in 2\pi c_1(M,J)\},
 \end{align}
 where $\lambda(\cdot)$ is the Perelman's entropy.  Then   for any initial metric  $\omega_0'\in 2\pi c_1(M,J)$
   flow $(M, J, \omega_t')$ of   (\ref{kr-flow}) is  uniformly  $C^\infty$-convergent to    $(M', J',\omega_{KR})$.
 \end{theo}

  We remark that ${\rm Aut}_0( M')$ is reductive \cite {Mat57} and (\ref{max-entropy-4}) holds  \cite{TZ4} if   $(M', J')$ admits a KE  metric.   Thus the assumption  in  Theorem \ref{stability-KR-smooth-complex} is  automatically satisfied.      The typical examples $(M, J)$  as in
  Definition \ref{deformation-j} are
   Mukai-Umemura's 3-folds, which have  studied by Tian, Donaldson \cite{Ti97, Do}.  The stability  of  (\ref{kr-flow}) on  those manifolds  was proved  by Sun-Wang, Wang by using Lojasiewiczj's inequality \cite{SW, Wang}. But if  $\omega_{KR}$ is not a KE  metric,
   the assumption  (\ref{max-entropy-4}) in  Theorem \ref{stability-KR-smooth-complex} is necessary   in general according to  a counter-example   $Gr_q(2,7)$  which can be deformed  to  a  horo-spherical manifold  found by Pasquier \cite{Pas}.  The latter can admit a KR soliton (non-KE)  by  a  result of Deltroix \cite{Del},   see Remark \ref{counter-example} for details.

  Theorem \ref{stability-KR-smooth-complex}  generalizes the result for the KE metric limit \cite{TZ4}, and also for the original complex manifold $(M,J)$ with admitting a   KR soliton while $J'$ is conjugate to $J$ \cite{TZZZ, DS}. In fact, in the latter cases, the convergence can be improved  in the sense of K\"ahler potentials.

 As an application of  Theorem  \ref{stability-KR-smooth-complex},   we prove the following uniqueness of KR solitons in sense of diffeomorphism orbits.

  \begin{cor}\label{uniqueness-orbit}  Let $\{\omega_i^1\}$  and $\{\omega_i^2\}$ be two  sequences of K\"ahler metrics in $2\pi c_1(M, J)$ which converge to  KR solitons  $(M_\infty^1,\omega_{KR}^1)$ and  $(M_\infty^2, \omega_{KR}^2)$ in sense of Cheeger-Gromov, respectively. Suppose that ${\rm Aut}_0( M_\infty^1)$ and  ${\rm Aut}_0( M_\infty^2)$ are both  reductive, and
  \begin{align}\label{max-entropy-10}\lambda(\omega_{KR}^1)=\lambda(\omega_{KR}^2)=\sup\{\lambda(g')|~\omega_{g'}\in 2\pi c_1(M,J)\}.
 \end{align}
  Then  $M_\infty^1$ is biholomorphic to $M_\infty^2$ and  $\omega_{KR}^1$ is isometric to $\omega_{KR}^2$.

   \end{cor}

 Corollary \ref{uniqueness-orbit} generalizes  Tian-Zhu's theorem  for  the uniqueness of KR solitons on a complex manifold \cite{TZ1},  and  it  is  also  a generalization of uniqueness result of Chen-Sun for  KE metric orbits \cite{CS} (also see \cite{LWX15}).  We also mention that the assumption  (\ref{max-entropy-10}) as in  Theorem \ref{stability-KR-smooth-complex} is necessary. A more generalization of Corollary \ref{uniqueness-orbit} for     limits   of singular KR soliton will be also discussed  in Section 6, see Remark \ref{generalization-smooth} and
 Theorem  \ref{uniqueness-orbit-singular}. We hope those results can be applied to study the modulo space of   singular KR solitons as done successfully for  singular KE metrics \cite{LWX15, LWX3, SS, ABHX19, BX}, etc.

  The paper is organized as follows.

    In Section 1, we   first recall  the partial $C^0$-estimate and  some local estimates for K\"ahler potentials  in
    \cite{WZ20} (cf. Lemma \ref{c3-phi}).
    Those estimate will  be generalized for  a sequence of evolved K\"ahler metrics of  (\ref{kr-flow}) with varied initial metrics
    in order ton prove a gap result for $Q$-Fano varieties in sense of Gromov-Hausdroff topology (cf. Proposition \ref{two-topology}).

      In Section 2,    We prove the uniqueness of soliton VFs associated to limits of (\ref{kr-flow}) with fixed
       initial metric (cf. Corollary \ref{vec}). The result can be regarded as an application of  Theorem \ref{WZ} (cf. Proposition \ref{continuity-vf}).  Corollary \ref{vec} will be generalized to the varied flow in Section 5 (cf. Proposition \ref{vector-gap}).

       In Section 3,   We first prove  a version of Luna's slice lemma in GIT-figure (cf. Lemma \ref{slice}), then as an application,  we prove the uniqueness of algebra structures of Q-Fano varieties $\tilde M_\infty$  associated to limits of (\ref{kr-flow}) with fixed
       initial metric under the assumption that  $ {\rm Aut}_0(\tilde M_\infty)$ is reductive for one of $\tilde M_\infty$ (cf. Corollary \ref{semiuni-KR-solions}).  Our proof  just  depends on the partial $C^0$-estimate and the  local estimates in Section 1.  The assumption will be removed in Section 4 as in \cite{CSW} (cf. (\ref{unique-alge-variety})). The proof of (\ref{unique-alge-variety}) in \cite{CSW} depends on  Chen-Wang's work on HT-conjecture,  thus our argument is a bit different to theirs.

        In Section 4,  we prove Theorem \ref{general-KR}  in case that $(\tilde M_\infty, \hat \omega_\infty)$ is a singular KE metric. Our proof is based on a result of Li for the lower bound estimate of Ding energy \cite{Li}. His method is to solve certain homogeneous MA equation associated to a  geodesic ray of K\"ahler metrics. Since this is an independent interest, we give a sketch proof of Li's result  (cf. Proposition \ref{low}).

       In Section 5,  we  represent a result of  Dervan-Sz\'ekelyhidi  that $L(\omega')$ is independent of $\omega'$  \cite{DS16} (cf. Proposition \ref{energy-level}).  Several applications of Proposition \ref{energy-level} will be also presented, see  Corollary \ref{mainThm-lambda},  Proposition \ref{vector-gap}.

  In Section 6, we prove the main results,  Theorem \ref{general-KR}  and Theorem \ref{stability-KR-smooth-complex} in this paper.
    Other more general results are also obtained (cf. Remark \ref{generalization-smooth} and Corollary \ref{uniqueness-orbit-singular}).

 In Section 7 (an appendix),  following  the proof of  Matsushima theorem for weak KE metrics in \cite{T2},    we  prove  the reductivity of ${\rm Aut}_0(\tilde M_\infty)$ for the limits of singular KE metrics of (\ref{kr-flow}) (cf. Proposition \ref{red}).  The result will be also generalized  for the limit of   singular KR solitons (cf. Proposition \ref{redsol}).

  \vskip3mm

  \noindent {\bf Acknowledgements.} The authors  would like to thank professor Gang Tian for inspiring conversations.  They also thank   Li for telling   us that the uniqueness of  algebraic structures of $\tilde M_\infty$ in  Theorem \ref{WZ} was  solved  by using non-archimedean geometry in  his recent joint paper with Han
 \cite{HL}.

  \section{Local convergence of KR flow}

 In this section, we first recall some notations and some technical results  in  \cite{WZ20}.
   Let $\{\omega_{i}\}$  be a   sequence of $(M, \omega_t)$ of  (\ref{kr-flow}). We  denote  $\{s_i^\alpha\}$  to be an  orthonormal  basis  of   $H^0(M,  K^{-l}_M, \omega_{i})$. Namely, $\{s_i^\alpha\}$ satisfies
   \begin{align}\label{inner-product}
 ( s_i^\alpha, s_i^\beta)=\int_M\langle   s_i^\alpha, s_i^\beta \rangle_{h_i^{\otimes l}}\omega_{i}^n,
 \end{align}
 where  $h_i$ is a hermitian metric  on   $K^{-1}_M$  such that ${\rm R }(h_i)=\omega_{i}.$
 %For simplicity, we set  $\omega_i=\omega_{t_i}$.

      Let   $\Phi_i: M\to \mathbb CP^N$ be the Kodaira embedding given by  $\{s_i^\alpha\}$ with  image  $\Phi_i(M)= \tilde M_i$.   By  \cite[Proposition 2.7]{WZ20},  $ \tilde M_i$  converges to a Q-Fano variety $\tilde M_\infty$ with klt singularities in algebra geometry  as in  Theorem \ref{WZ}.  We write $[\tilde M_i]$ and $[  \tilde M_\infty]$ as Chow points in $W_N=\mathbb  P(\mathbb C(M_{(n+1)\times (N+1)}), \mathbb C)$, where  $\mathbb C(M_{(n+1)\times (N+1)}$ denote a linear space of  $(n+1)$-multiple  homogeneous polynomials  with variables $N+1$.   We note that  ${\rm SL}(N + 1;\mathbb C)$  acts  naturally on $[\tilde M_i]$.
      Then the convergence of  $\tilde M_i$ means that  $[\tilde M_i]$ converges to $[  \tilde M_\infty]$. The latter is also equivalent to that there are $g_i\in   {\rm SL}(N + 1;\mathbb C)$ such that
      \begin{align}\label{group-g}\lim_i g_i\cdot [\tilde M]=[  \tilde M_\infty],
      \end{align}
      where  $\tilde M\subset \mathbb CP^N$ is  regarded as a complex  submanifold  by the Kodaira embedding of $M$.

 By (\ref{inner-product}) and  the continuity of  flow (\ref{kr-flow}), the  Kodaira embedding $\Phi_t$ generates a continuous path
 $[\tilde M_{t}]$ $(t\in (0,\infty))$ with the  variable  $t$  in $W_N$. Since the limit $\tilde M_\infty$ of $\tilde M_{t}$  may be different  for a different subsequence $\{t_i\}$, we introduce a set $\mathcal C_0$  which  consists of all  possible limits $\tilde M_\infty$,  and set
 $$[\mathcal C_0]=\{[\tilde M_\infty]|~ \tilde M_\infty\in \mathcal C_0\}.$$

 The following is an element result.
   \begin{lem}\label{connected}
    $[\mathcal C_0]$ is a compact connected subset in  $W_N$.
   \end{lem}

   \begin{proof}
   The compactness follows from the diagonal argument  since  $W_N$ is compact, so  we need to  prove the connectedness. On the contrary, we assume there are two closed sets  $U$ and $V$
    such that $[\mathcal C_0]= U\bigcup V$ and
    $${\rm dist}_{W_N}(U, V)\geq \delta>0.$$
     Then there are two  points $[\tilde M^1_\infty]\in U$ and   $[\tilde M^2_\infty]\in V$ and two  sequences $\{[\tilde M_{t_i}]\}$ and $\{[\tilde M_{s_i}]\}$, respectively,  such that
     $$\lim_i [\tilde M_{t_i}]\to [\tilde M^1_\infty]~{\rm and}~\lim_i [\tilde M_{s_i}]\to [\tilde M^2_\infty].$$
     Without loss of generality, we  may   $t_i<s_i<t_{i+1}$. Thus
     $${\rm dist}_{W_N}([\tilde M_{s_i}], V)\to 0~{\rm and}~ {\rm dist}_W([\tilde M_{t_i}], V)\ge \frac{\delta}{2}, ~{\rm as}~i\to\infty. $$
      The latter means  that there exists $t'_i\in (s_i, t_{i+1})$ such that
       $${\rm dist}_{W_N}([\tilde M_{t'_i}], V)= \frac{\delta}{4}.$$
       Hence  by the triangle inequality, we also get
       $${\rm dist}_{W_N}([\tilde M_{t'_i}], U)\geq \frac{3\delta}{4}.$$
    On the other hand,    by Theorem \ref{WZ}, $\tilde M_{t'_i}$
      converges to  a limit $\tilde M_\infty'$ in  $\mathcal C_0$.  Combining  the above two relations, we get
      $${\rm dist}_{W_N}([\tilde M_\infty'], V)= \frac{\delta}{4} ~{\rm and}~{\rm dist}_{W_N}([\tilde M_\infty'], U)\ge\frac{3\delta}{4}.$$
      This implies that  $ [\tilde M_\infty']$ is not contained in $ U\bigcup V$, which is impossible. The lemma  is proved.

   \end{proof}

  Let   $\tilde M_\infty\in [\mathcal C_0]$. We
 choose  an exhausting open sets $\Omega_\gamma\subset  \tilde M_\infty$.  Then   there are  diffeomorphisms $\Psi_\gamma^i: \Omega_\gamma\to \tilde M_i$ such that
   the curvature of $ \omega_{FS}|_{\tilde {\Omega}_\gamma^i}$ is $C^k$-bounded, i.t. it is   uniformly independently of $i$ , where  $\tilde {\Omega}_\gamma^i=\Psi_\gamma^i(\Omega_\gamma)$.

  For simplicity,  we let $\tilde\omega_i=\frac{1}{l}\omega_{FS}|_{\tilde M_i}$.  Then we can write
   \begin{align}\label{psi-s}(\Phi_i^{-1})^*\omega_{t_i+s}=\tilde\omega_{i}+\sqrt{-1}\partial\bar\partial \psi^s_i,   ~{\rm in} ~\tilde M_i, \forall s\in [-1,1].
   \end{align}
    In  Proposition 3.2 in \cite{WZ20},  it was proved that

   \begin{lem}\label{c3-phi}
   There exist  constants $A, C_\gamma, A_\gamma$ such that for $s\in [-\frac{1}{2},1]$,
   \begin{align}
   &| \psi^s_i|\le A, ~{\rm in} ~\tilde M_i, \label{c0-estimate}\\
   &C_\gamma^{-1}\tilde\omega_{i}\le(\Phi_i^{-1})^*\omega_{t_i+s}\le C_\gamma \tilde\omega_{i}, ~{\rm in}~\tilde \Omega_\gamma^i,
   \label{c2-estimate} \\
   &\| \psi^s_i\|_{C^{k,\alpha}(\tilde {\Omega}_\gamma^i)} \le A_\gamma(k).\label{c3-estimate}
   \end{align}
  \end{lem}

  Let $\omega_t=\omega_0+\sqrt{-1}\partial \bar\partial \phi_t$ be the solution of (\ref{kr-flow}). Then
  $${\rm Ric}\,(\omega_t)-\omega_t=\sqrt{-1}\partial \bar\partial (-\dot \phi_t ).$$
  Thus  $-\dot \phi_t $ is a Ricci potential of $\omega_t$.
   The following estimates are due to G. Perelman. We refer the reader to
  \cite{ST} for their proof.

  \begin{lem}\label{lem:perelman-1}  There are constants $c>0$ and $C>0$ depending only on the initial metric $\omega_0$ such that the following is true:

  1) ${\rm diam}(M,\omega_t)\le C$, ~
  ${\rm vol}(B_r(p),\omega_t )\ge c r^{2n}$.

  2) For any $t\in (0,\infty)$, there is a constant $c_t$ such that  $h_t=-\dot \phi_t +c_t$ satisfies
  \begin{align}\label{h-t-estimate}  \|h{_t}\|_{C^0(M)}\le C,
  ~\|\nabla h_{t}\|_{\omega_{t}} \le C,
  ~ \|\Delta h_{t}
  \|_{C^0(M)} \le C.
  \end{align}
   \end{lem}

 We also need to consider the modified K\"ahler metrics $\eta_t$ of $\omega_t$ used  in \cite{ZhK, WZ20},  which is  a solution of
  \begin{align}\label{eta-equation}{\rm Ric}\,(\eta_t)=\omega_t.
  \end{align}
  Write  $\eta_t= \omega_t+\sqrt{-1}\partial \bar \partial \kappa_t$. Then   we have a solution $\kappa_t$ of complex MA equation,
  \begin{align}\label{modified-metric}
  (\omega_t+\sqrt{-1}\partial \bar \partial \kappa_t)^n=e^{h_t}\omega_t^n,~\sup_M\kappa_t=0,
  \end{align}
 where $h_t$ is a Ricci potential of $\omega_t$ as in  (\ref{h-t-estimate}). By using the iteration method,   $\kappa_t$ is uniformly bounded.  Moreover,  similar to  relations  (\ref{c2-estimate}) and (\ref{c3-estimate}) in Lemma \ref{lem:perelman-1},  we have
   \begin{align} &C_\gamma^{-1}\tilde\omega_{i}\le(\Phi_i^{-1})^*\eta_{i}\le C_\gamma \tilde\omega_{i}, ~{\rm in}~\tilde \Omega_\gamma^i,
   \label{c2-estimate-1} \\
   &\| \kappa_i\|_{C^{k,\alpha}(\tilde {\Omega}_\gamma^i)} \le A_\gamma(k).\label{c3-estimate-1}
   \end{align}

 The limit set    $[\mathcal C_0]$ in Lemma  \ref{connected} is for  sequences  of  KR flow with a fixed initial metric.  We shall extend
 the lemma to  general limits  of   sequences  from   KR flows  with varied  initial metrics.  To distinguish the  flow (\ref{kr-flow}) for different  initial metrics, we denote $(M, \omega_t'=\omega_t^\alpha)$ to be a  KR flow  with an initial metric  $\omega_0'=\omega^\alpha\in 2\pi c_1(M,J)$.  Since  Perelman's estimate for the  metric
 $\omega_{t}'$  in  Lemma  \ref{lem:perelman-1} depends only on the initial metric $\omega_0'$,  the partial $C^0$-estimate proved \cite{ZhK} holds for all $\omega_t'$. Namely,  there exists $l'=l(M, \omega_0')\in \mathbb N^*$ and $b=b(M,\omega_0')>0$ such that
 \begin{align}\label{unip}
 \rho_{l'}(M,\omega_t')=\sum_k \| s_k'\|_{\omega_t'}  \geq b,
 \end{align}
   where  $\{ s_k'\}$ is an  ortho-normal basis of $H^0(M, K_M^{-l'}, \omega_t')$ which gives
      a Kodaira embedding into a   fixed projective space $\mathbb CP^{N'}$.

     We may assume
     \begin{align}\label{a-condition}\|\omega_0'-\omega_0\|_{C_{CG}^2(M)}\le A,
     \end{align}
      where $\|\cdot\|_{C_{CG}^2(M)}$ denotes the $C^2$-norm in Cheeger-Gromov topology.
 Then  Lemma \ref{c3-phi} still holds.  By  the argument  in the proof of  \cite [Proposition 3.2]{WZ20},  for any sequence $\{ \omega_{t_i}^{\alpha_i}\}$,  there is a Q-Fano variety $\tilde M_\infty'$ such that   $\omega_{t_i}^{\alpha_i}$ is locally  $C^\infty$-convergent to a  K\"ahler metric   on  ${\rm Reg}(\tilde M_\infty')$.  We note that  $\omega_{t_i}^{\alpha_i}$  has also a Gromov-Hausdroff limit $(M_\infty', \omega_\infty')$ by the Perelman's estimate in  Lemma  \ref{lem:perelman-1}-1).

 Let $(M_\infty, \omega_\infty)$   be  the Gromov-Hausdroff limit of  $\{\omega_i\}$ as in Theorem \ref{WZ}.  The following is the main result in this section.

   \begin{prop}\label{two-topology}Let  $(M_\infty, \omega_\infty),    (M_\infty', \omega_\infty'), \tilde M_\infty$ and $ \tilde M_\infty'$ be  Gromov-Hausdroff limits and algebraic variety  limits for sequences  $\{\omega_i\}$  and  $\{ \omega_{t_i}^{\alpha_i}\}$ as above,  respectively.  Suppose
      \begin{align}\label{small-GH-0}{\rm dist}_{GH}(  (M_\infty, \omega_\infty),    (M_\infty', \omega_\infty'))\le \epsilon.
   \end{align}
   Then  there is a $g\in {\rm SL}(N' + 1;\mathbb C)$ such that
    \begin{align}\label{small-algebra}{\rm dist}_W  ([\tilde M_\infty'],  g\cdot [\tilde M_\infty])\le \delta(\epsilon)\to 0, ~{\rm as}~\epsilon\to 0 .
    \end{align}
   \end{prop}

  \begin{proof}We use the contradiction argument and suppose that there are    sequences  $\{\omega_{t_i,k}^{\alpha_i}\}$ with their
    Gromov-Hausdroff limits and algebraic variety limits  $(M_{\infty,k}', \omega_{\infty,k}'), \tilde M_{\infty,k}'$, respectively,  such that
     \begin{align}\label{small-GH-1}{\rm dist}_{GH}(  (M_{\infty,k}', \omega_{\infty,k}'),    (M_\infty, \omega_\infty))\to 0,~{\rm as}~\alpha\to\infty
   \end{align}
    and
      \begin{align}\label{small-algebra-2} {\rm dist}_{W}  (  [\tilde M_{\infty,k}'],  {\rm SL}(N' + 1;\mathbb C)\cdot [\tilde M_\infty])\ge c_0>0,  ~\forall~\alpha>>1 .
    \end{align}
    Then by (\ref{small-GH-1}) we can take a diagonal subsequence   $\{\omega_{i_k}^{\alpha_i}\}$ such that
      \begin{align}\label{small-GH-2}{\rm dist}_{GH}(  (M, \omega^{\alpha_i}_{i_k}),    (M_\infty, \omega_\infty))\to 0,~{\rm as}~i\to\infty.
   \end{align}
   Namely,    $ (M_\infty, \omega_\infty)$ is also  the   Gromov-Hausdroff limit of  $\omega^{\alpha_i}_{i_k}$. On the other hand,   as in the proof of  \cite [Proposition 3.2]{WZ20},  there is a Q-Fano variety $\bar M_\infty$ such that   $\omega^{\alpha_i}_{i_k}$ is locally  $C^\infty$-convergent to a  K\"ahler metric  $\bar \omega_\infty$  on  ${\rm Reg}(\bar M_\infty)$. Moreover,
     \begin{align}\label{isometric}\overline {( {\rm Reg}(\bar M_\infty), \bar\omega_\infty)}\cong (M_\infty, \omega_\infty).
     \end{align}

     By Theorem \ref{WZ}, we know
    \begin{align}\label{isometric-0}  \overline {( {\rm Reg}(\tilde M_\infty), \hat \omega_\infty)}\cong (M_\infty, \omega_\infty).
    \end{align}
     Moveover, by Lemma \ref{regular-set} below,
     $(M_\infty, \omega_\infty)\setminus (\hat \omega_\infty, {\rm Reg}(\tilde M_\infty))$ is a singular set.
    Thus by  (\ref{isometric}), any restricted metric of $\bar \omega_\infty$ on an  open  set $U\subset  {\rm Reg}(\bar M_\infty))$ must be isometric to one of restricted metric $\hat \omega_\infty$ on some  open  set in  ${\rm Reg}(\tilde M_\infty)$. This means that  $\bar \omega_\infty$ is a singular KR soliton  on ${\rm Reg}(\bar M_\infty)$. Furthermore,    as in the proof of  Theorem \ref{WZ} \cite{WZ20}, $\bar M_\infty$ has $klt$-singularities and   $\bar \omega_\infty$ can be extended to  a singular  KR soliton  on $\bar M_\infty$ with a $L^\infty$  K\"ahler potential.

        By (\ref{isometric}) and (\ref{isometric-0}),  the above argument also implies that
    \begin{align}\label{open-iso}( {\rm Reg}(\bar M_\infty), \bar\omega_\infty)\cong ( {\rm Reg}(\tilde M_\infty), \omega_\infty).
    \end{align}
      Thus, after a diffeomorphism, we may assume that  ${\rm Reg}(\bar M_\infty)$ is biholomorphic to  ${\rm Reg}(\tilde M_\infty)$
    and both of  K\"ahler-Ricci solitons $\bar\omega_\infty$ and $\tilde\omega_\infty$ are same.
    We claim  that  there is a $g\in {\rm SL}(N + 1;\mathbb C)$ such that
    \begin{align}\label{conguarate}
    \bar M_\infty=g\cdot\tilde M_\infty.
    \end{align}

   let $\eta_i$ and  $\eta_{i_\alpha} $ be a solution of (\ref{eta-equation}) for    $\omega_i$ and   $ \omega^{\alpha_i}_{i_k}$, respectively.
    Then by (\ref{c2-estimate-1}) and  (\ref{c3-estimate-1}),  $\{\kappa_{i}\}$ and $\{\kappa_{i_\alpha}\}$
   converge to  uniformly bounded potentials $\kappa_\infty$ and $\kappa_\infty'$ on  ${\rm Reg}(\tilde M_\infty)$, respectively,  both of  which  satisfy
   \begin{align}\label{yau-equation-limit}
  (\omega_\infty+\sqrt{-1}\partial \bar \partial \kappa)^n=e^{\tilde h_\infty-\psi_\infty} (\tilde\omega_\infty)^n,~{\rm in}~\tilde M_\infty.
  \end{align}
   By the uniqueness of bounded solutions,  we get
   $$\kappa_\infty=\kappa_\infty'$$
   and so
   $$\eta_\infty=\omega_\infty+\sqrt{-1}\partial\bar\partial \kappa_\infty=\omega_\infty+\sqrt{-1}\partial\bar\partial \kappa_\infty'=\eta_\infty'.$$
   Since,  by \cite[(4.30)]{WZ20},
     $$\overline {(\eta_\infty, {\rm Reg}(\bar M_\infty))}=(X, d_\infty)~{\rm and}~\overline {(\eta_\infty', {\rm Reg}(\bar M_\infty))}=(X', d_\infty')$$
     where $(X, d_\infty)$ and  $(X', d_\infty')$ are the   Gromov-Hausdroff limits of $\eta_i$ and  $\eta_{i_\alpha}$, respectively,
        $$(X, d_\infty)=(X', d_\infty').$$

      Let $\{s^k_i\}$  and   $\{{s'}^k_i\}$ be   orthonormal  bases   of   $H^0( M,  K^{-l'}_M, \eta_i)$  and  $H^0( M,  K^{-l'}_M, \eta_{i_\alpha})$, respectively.  Then as in Lemma 4.4 in \cite{WZ20},  $\{s^k_i\}$  and   $\{{s'}^k_i\}$  converge to  collections of $\{s^k_\infty\}$    and    $\{ {s'}^k_\infty\}$ of unitary and orthogonal holomorphic sections   in    $H^0( {\rm Reg}( \tilde M_\infty),  K^{-l'}_{{\rm Reg}( \tilde M_\infty)}, \eta_\infty)$, respectively,
       where  $H^0( {\rm Reg}( \tilde M_\infty),  K^{-l'}_{{\rm Reg}( \tilde M_\infty)}, \eta_\infty)$ denotes the space of bounded holomorphic sections on  $({\rm Reg} (\tilde M_\infty),  K^{-l'}_{{\rm Reg}( \tilde M_\infty)})$.
   Moveover,
   \begin{align}\label{two-maps}\overline{\Phi_\infty ({\rm Reg}(\tilde M_\infty))}=\tilde M_\infty~{\rm and}~\overline{\bar\Phi_\infty ({\rm Reg}(\tilde M_\infty))}=\bar M_\infty,
   \end{align}
    where  $\Phi_\infty$ and  $\bar\Phi_\infty$ are Kodaira embeddings  given by   $\{s^k_\infty\}$ and     $\{ {s'}^k_\infty\}$, respectively.
 On the other hand,  according to the proof of normal variety of $\tilde M_\infty$ in \cite{LS},  the dimension of $H^0( {\rm Reg}(\tilde M_\infty),  K^{-l'}_{{\rm Reg}( \tilde M_\infty)})$ is same as one of  $H^0(M,  K^{-l'}_M)$. Thus there is a $g\in {\rm SL}(N' + 1;\mathbb C)$ such that
 $$\{s^k_\infty\} =g\cdot \{{s'}^k_\infty\}.$$
 Hence, by (\ref{two-maps}), we finally get  (\ref{conguarate}).

 By  Theorem \ref{WZ} together with  (\ref{conguarate}), we see that there are $g_i\in {\rm SL}(N' + 1;\mathbb C)$ such that
 $$[\tilde \Phi_i(M)]\to  g\cdot[\tilde M_\infty],$$
 where $\tilde \Phi_i$ is a   Kodaira embedding  given by an  ortho-normal  basis in $\{\tilde s^k_i\}$ in $H^0( M,  K^{-l'}_M, \omega_{i_\alpha}^\alpha)$.
 On the other hand, for a fixed $\alpha$ we may take $i_\alpha$ in (\ref{small-GH-1}) such that
 $${\rm dist}  (  [\tilde \Phi_i(M)], [\tilde  M_\infty^\alpha])\le \delta(i_\alpha)\to 0, ~{\rm as}~i_\alpha\to \infty.$$
 Thus
 $${\rm dist}  (   [\tilde  M_\infty^\alpha], g\cdot [\tilde M_\infty])\le 2\delta(i_\alpha),$$
 which is a contradiction with (\ref{small-algebra-2}). The proposition is proved.

 \end{proof}

 By  \cite[Theorem 1.2]{Bam}, the Gromov-Hausdorff limit $(M_\infty, \omega_\infty)=\lim_{i\rightarrow \infty} (M,\omega_i)$ has a decomposition into regular part $\mathcal R$ and singular part $\mathcal S$. The regular part is an open manifold and the restriction of $\omega_\infty$ on the regular part is smooth metric. Moreover, $\omega_i$ converges to $\omega_\infty$ smoothly on $\mathcal R$. We will compare  the regular parts of $M_\infty$ and $\tilde M_\infty$. Since $\Phi_i: M\rightarrow \tilde M_i$ are uniformly Lipshcitz \cite{WZ20}, $\Phi_i$ converges to some map $\Phi_\infty: M_\infty\rightarrow \tilde M_\infty$.

 \begin{lem}\label{regular-set}
 $\Phi_\infty(\mathcal R)={\rm Reg}(\tilde M_\infty)$ and the restriction of $\Phi_\infty$ on $\mathcal R$ is one-to-one.
 \end{lem}

 \begin{proof}
 At first, assuming that $x=\lim_{i\rightarrow \infty} p_i$ and $z=\Phi_\infty(x)\in {\rm Reg}\,(\tilde M_\infty)$, by Theorem \ref{WZ}, there is a neiborhood $V$ of $z$ and neighborhood $V_i$ of $p_i$ such that $\omega_i$ on $V_i$ converges smoothly to $\bar \omega_\infty$ on $V$. Choosing a smaller convex neighborhood $\tilde V_i$ of $p_i$, the limit of $\tilde V_i$ will be an open manifold containing $x$. Thus we get $x\in \mathcal R$. This implies that ${\rm Reg}\,(\tilde M_\infty)\subseteq \Phi_\infty(\mathcal R)$.

 Next we prove the inverse part.  By \cite{Zh}, the Sobolev constants of $\omega_i$ is uniformly bounded. So from Lemma 3.1 in \cite{JWZ}, the $L^\infty$ and gradient estimates of holomorphic sections of $H^0(M,K_M^{-q},\omega_i)$ hold for any positive integer $q$. By  Lemma \ref{lem:perelman-1}, the gradient of the Ricci potential of $\omega_i$ is uniformly bounded \cite{WZ20}. Thus  we have the  estimate of eigenvalue of operator $\bar\partial$ of $K_M^{-q}$ with respect ro metric $\omega_i$ for $q\geq 4nC^2$ by Lemma 3.3 in \cite{JWZ}.  Combined with the result of Bamler that  $\mathcal R$ is  the local $C^\infty$-limit of  $\omega_i$ in  the Cheeger-Gromov topology,  the proof of partial $C^0$ estimate  can be applied at the regular point of $M_\infty$  (cf. \cite{Ti13}, \cite{JWZ}) in the following.

 For any $x\in \mathcal R$, there is an open neighborhood $U$ of $x$ such that $\omega_\infty$ is a smooth metric on $U$. Choose intger $q$ large enough such that for $r=\frac{1}{q}>0$, there is a diffeomorphism $\phi:B_x(r)\rightarrow B_0(1)\subset \mathbb C^n$ satisfying $\phi^*(l^2\omega_\infty)$ is close to $\omega_{Euc}$ in $C^3$ sense. Denote the limit line bundle of $K_M^{-1}$ on $\mathcal R$ by $K_{\mathcal R}^{-1}$.  Becasue $\pi_1(B_x(r))$ is trvial, by Lemma 4.5 in \cite{JWZ}, without raising power, there is a section $\psi$ of $K_{\mathcal R}^{-q}\otimes L_0$ with $|D\psi|$ small, where $L_0$ is the trivial bundle $B_0(1)\times \mathbb C$ endowed with Hermitian metric $e^{-|z|^2}$. This is equivalent to a bundle map $\psi: K^{-q}_{\mathcal R}\rightarrow B_0(1)\times \mathbb C$ with $|D \psi|$ small. By Theorem 1.2 in \cite{Bam}, there is a sequence of points $p_i$ and diffeomorphism $\phi_i: (B_x(r),\omega_\infty)\rightarrow (B_{p_i}(r),\omega_{i})$ such that $\phi_i^*\omega_{i}$ converges to $\omega_\infty$ smoothly.  Let $\zeta: \mathbb R\rightarrow \mathbb R$ be a cut-off function which satisfies:
 $$ \zeta(t)=1, \text{ for } t\leq\frac{1}{2}; \zeta(t)=0, \text { for } t\geq 1; |\zeta'(t)|\leq 2. $$
  Let $z^j$ be the coordinates functions on $\mathbb C^n$, choosing $R$ large enough, then $\tau^j_i=\phi_i^{*}(\psi^*(\zeta(\frac{d(0,\cdot)}{R})z^j))$ is a smooth section of $K^{-q}_M$ with $|\bar\partial \tau_i|_{L^2}$ small. We can assume that $B_x(Rr)$ is still contained in $U$. By Lemma 3.3 in \cite{JWZ}, we can solve the $\bar\partial$ equation:
 $$\bar\partial \sigma^j_i=\bar\partial\tau^j_i,$$ such that $|\sigma^j_i|_{L^2}$ is small. Then $s^j_i=\tau^j_i-\sigma^j_i(1\leq j\leq n)$ are holomorphic sections. By the local regularity of ellipitic equation, both $|\sigma^j_i|$ and $|\nabla \sigma^j_i|$ are small near $p_i$. So the Jacobian of $(s_i^1,...,s_i^n)$ at $p_i$ is uniformly bounded from below. By Proposition A.2 in \cite{JWZ}, for $q\geq (n+2)l+2$, $H^0(M, K_M^{-q},\omega_{i})$ is contianed in $H^0(M, K_M^{-l},\omega_{i})\otimes H^0(M, K_M^{-(q-l)},\omega_{i})$. So replacing $l$ by $(n+3)l$, the tangent map of $\Phi_i$ at $p_i$ has full rank. Since the Jacobian of $(s_i^1,...,s_i^n)$ at $p_i$ is uniformly bounded from below, the tangent map of $\Phi_\infty$ at $x$ also has full rank. Thus $\Phi_\infty(x)\in {\rm Reg}\,(\tilde M_\infty)$ and $\Phi_\infty(\mathcal R)={\rm Reg}(\tilde M_\infty)$. The injectivity can be proved as Proposition 8.2 in \cite{JWZ}. The lemma is proved.
 \end{proof}

 \begin{rem}\label{soliton-case} To get  (\ref{open-iso}), we shall use   Lemma \ref{regular-set}, whose proof  depends on a deep result of Bamler  \cite[Theorem 1.2]{Bam}.   Lemma \ref{regular-set} is also proved in  \cite{WZ20}  for  two special cases of $( M_\infty,\omega_\infty)$:  $( M_\infty,\omega_\infty)$  is  a smooth KR soliton   (cf. Lemma \ref{smoothcase-lemma} below), or   a singular KE metric
 (cf.  (4.31) in \cite{WZ20}).
 \end{rem}

   \section{Uniqueness of soliton VFs }

   For a   Q-Fano variety $\tilde M_\infty$,  we   denote ${\rm Aut}(\tilde M_\infty)$  to be a subgroup of ${\rm SL}(N+1;\mathbb C)$ whose element as  an action fixes  $M_\infty $.   In this section,  we show that  the soliton  VF associated  to $\tilde M_\infty\in \mathcal C_0$ is unique.  First we have

   \begin{lem}\label{vec1}
   Let $\tilde M^1_\infty$ and $\tilde M^2_\infty\in \mathcal C_0$ be two  limits  and $X^1, X^2$ be the corresponding soliton VF. Assume that there is a complex torus $T\subset {\rm  Aut}(\tilde M_\infty^1)\bigcap {\rm Aut}(\tilde M_\infty^2)$ such that $X^1, X^2\in  {\rm Lie}(T)$. Then $X^1=X^2$.
   \end{lem}

   \begin{proof}
   Denote the restriction of $\frac{\omega_{FS}}{l}$ on $\tilde M_\infty^i$ by $\tilde \omega_\infty^i$ and the potential of $Z\in {\rm  Lie}(T)$ with respect to $ \tilde \omega_\infty^i$ by $ \theta^i_Z$ ( $i=1,2$), which satisfies
   $$i_Z( \tilde\omega_\infty^i)=\sqrt{-1}\bar \partial  \theta^i_Z,~\int_ {M_\infty^i}  \theta^i_Z e^{h^i} ( \tilde\omega_\infty^i)^n=0,$$
   where $h^i$  is a Ricci potential of  $\tilde\omega_\infty^i$.
   As in the proof of Lemma 2.1 in \cite{TZ2}, consider the function $f^i$ on ${\rm Lie}(T)$ defined by
    $$f^i(Z)=\int_{\tilde M^i_\infty} e^{ \theta^i_Z}(\tilde \omega^i_\infty)^n.$$
     Then by the proof of Proposition \ref{center}, $f^i$ is a proper function on ${\rm Lie}(T)$ and $X^i$ is the unique point such that $\nabla f^i(X^i)=0$. We claim that $f^1=f^2$ as a function on ${\rm Lie}(T)$. By  the equivariant Riemann-Roch Theorem (cf. \cite{WZZ}), we have
 $${\rm Tr }(e^{\frac{1}{k}X})|_{K^{-k}_{\tilde M^i_\infty}}=\int_{\tilde M^i_\infty} e^{k(\tilde\omega^i_\infty+\frac{1}{k}\theta_Z^i)}(\tilde\omega_\infty^i)^n+O(k^{n-1}).$$
The leading term on the right hand side is $k^n\int_{\tilde M^i_\infty}  e^{\tilde \theta^i_Z}\frac{(\tilde \omega^i_\infty)^n}{n!}$. Since $K^{-1}_{\tilde M^i_\infty}$ is the restriction of $\frac{1}{l}\mathcal O(1)$, ${\rm Tr }(e^{\frac{1}{k}X})|_{K^{-lk}_{\tilde M^i_\infty}}$ are the same for $i=1,2$. The claim is proved and it follows that $X^1=X^2$.
   \end{proof}

   Fix a maximal torus $(\mathbb C^*)^N\subset {\rm SL}(N+1;\mathbb C)$. Let $X$ be a soliton VF of  $(\tilde M_\infty,\hat\omega_\infty) \in \mathcal C_0$.   By \cite[Lemma 4.4]{WZ20},  we know that $X$ can be extended to an element in ${\rm sl}(N+1;\mathbb C)$.  Then there is a maximal complex torus $T\subset {\rm Aut}_r(\tilde M_\infty)$ such that $X\in {\rm Lie}(T)$.  Since $T$ can be conjugated to lie inside $(\mathbb C^*)^N$ by a unitary matrix, $X$ is mapped to a VF in ${\rm Lie}(\mathbb C^*)^N\cong \mathbb C^N$. The image of $X$ is unique up to the action of Weyl group $S_N$. Thus we get a map $vec: \mathcal C_0\rightarrow \mathbb C^N/S_N$. In the following, we show that the map  induces  a natural topology of $\{X\}$ by $[ \mathcal C_0]$.

 Let $\{[\tilde M_\infty^i]\}$ be  a sequence in  $[\mathcal C_0]$ which converges to $[\tilde M_\infty']$.  Let $X^i$ and  $X'_\infty$ be soliton VFs associated to $(\tilde M_\infty^i, \hat \omega_\infty^i)$ and $(\tilde M_\infty', \hat\omega_\infty')$, respectively.  Then we prove

   \begin{prop}\label{continuity-vf}$X^i$ converges to $X_\infty'$ as $[\tilde M_\infty^i]$ converges to $[\tilde M_\infty']$.
   \end{prop}

   \begin{proof}
   Let $\tilde\omega_\infty^i$ be the restriction of $\frac{1}{l}\omega_{FS}$ on $\tilde M_\infty^i$. Writing $\omega_\infty^i=\tilde \omega_\infty^i+\sqrt{-1}\partial \bar\partial \phi_\infty^i$, we have
   \begin{align}\label{limitsol}
   (\tilde \omega_\infty^i+\sqrt{-1}\partial \bar\partial \phi_\infty^i)=e^{\tilde h^i_\infty-X^i(\phi_\infty^i)-\tilde \theta_\infty^i-\phi_\infty^i}(\tilde \omega_\infty^i)^n,
   \end{align}
   where $\tilde\theta_\infty^i$ is the potential of $X^i$ with respect to $\tilde \omega_\infty^i$. By the partial $C^0$ estimate, we have
   $$\tilde \omega_\infty^i\leq C(\tilde \omega_\infty^i+\sqrt{-1}\partial \bar\partial \phi_\infty^i),~ |\phi_\infty^i|\leq C$$
    for a uniform constant $C>0$.  On the other hand, by Lemma \ref{lem:perelman-1},  we know that $|X^i|_{\omega_\infty^i}$ is uniformly bounded. Since each $X^i$ can be extended to an element in ${\rm sl}(N+1;\mathbb C)$ \cite[Lemma 4.4]{WZ20},  $|X^i|_{\omega_{FS}(\mathbb CP^N)}\le C$.  Thus by taking a  subsequence, we get a limit $X'$ such that  $\lim_{i\rightarrow \infty}X^i=X'$.

    As in Section 2, we can choose exhausting open sets $\Omega_\gamma\subset {\rm Reg}(\tilde M_\infty)$ and $\Omega_\gamma^i\subset  {\rm Reg}(\tilde M_\infty^i)$ converging to $\Omega_\gamma$. By $\tilde \omega_\infty^i\leq C(\tilde \omega_\infty^i+\sqrt{-1}\partial \bar\partial \phi_\infty^i)$ and boundedness of right hand side in (\ref{limitsol}), we get $|\phi_\infty^i|_{C^{1,\alpha}(\Omega_\gamma^i)}\leq C_\gamma$ for  some constant $C_\gamma$. In particular,  $|X^i(\phi^i_\infty)|_{C^{\alpha}(\Omega_\gamma^i)}$ is bounded. By Evans-Krylov's theory,  there is a constant $A_\gamma$ such that $|\phi_\infty^i|_{C^{3,\alpha}(\Omega_\gamma^i)}\leq A_\gamma$.  Then  $\phi_\infty^i$ converges locally  to a function $\phi_\infty'$ on $M_\infty$. Denoting the restriction of $\frac{1}{l}\omega_{FS}$ by $\tilde\omega_\infty$ and the potential of $X'$ with respect to  $\tilde\omega_\infty$ by $\tilde \theta'_\infty$, we have
    \begin{align}
   (\tilde \omega_\infty+\sqrt{-1}\partial \bar\partial \phi'_\infty)=e^{\tilde h_\infty-X'(\phi'_\infty)-\tilde \theta'_\infty-\phi'_\infty}(\tilde \omega_\infty)^n.
   \end{align}
  Thus  $\tilde \omega_\infty+\sqrt{-1}\partial \bar\partial \phi'_\infty$ is a KR soliton on $\tilde M_\infty$. By the uniqueness of KR  soliton, there is a $\sigma\in {\rm Aut}(\tilde M_\infty)$ such that $\sigma^*\omega_\infty'=\tilde \omega_\infty+\sqrt{-1}\partial \bar\partial \phi'_\infty$. Since the restriction of standard sections of $\mathcal O(1)$ are othor-normal basis of $H^0(\tilde M_\infty, K_{\tilde M_\infty}^{-l})$ with respect to $\omega_\infty$ and $\tilde \omega_\infty+\sqrt{-1}\partial \bar\partial \phi'_\infty$, $\sigma$ is a unitary matrix. Hence,  we have $\sigma_* X'= X_\infty'$ and $\lim_{i\rightarrow \infty}vec(\tilde M_\infty^i)=vec(\tilde M_\infty)$.
   \end{proof}

   \begin{cor}\label{vec}
   The image of $vec$ is a single point, which means that the soliton VF for all the limits is unique up to conjugation of $U(N+1;\mathbb C)$.
   \end{cor}

   \begin{proof}
   For any $M_\infty \in\mathcal C$, the maximal torus $T$ generated by the sotilon VF $X$ can be conjugated to a subtorus in $(\mathbb C^*)^N$. If two such subtorus are the same, by Lemma \ref{vec1}, the two soliton VF  are equal. Because there are countably many subtorus of $(\mathbb C^*)^N$, the image of $vec$ is countable. By Lemma \ref{connected}, the image of $vec$ is connected and must be a single point.
   \end{proof}

  \section{ GIT-figure and  uniqueness of $\tilde M_\infty$ }

 In this section, we   use GIT to prove the uniqueness of algebraic structure $\tilde M_\infty$  in case of reductive ${\rm Aut}_0(\tilde M_\infty)$.  We first prove  a version of Luna's slice lemma.

  \begin{lem}\label{slice}
  Let $V$ be a representation vector space of ${\rm SL}(N+1;\mathbb C)$. Assume that the identity component of the stabilizer of $[v_0]\in \mathbb P(V)$, denoted by $G$  is reductive. Then there is a projective subspace $\mathbb P_1$ containing $[v_0]$ and a neighborhood $U$ of $[v_0]$ in $\mathbb P(V)$ such that the following holds:

  1)  $\forall x \in U$, ${\rm SL}(N + 1;\mathbb C)\cdot x$ intersects with $\mathbb P_1; $

   2) every component of ${\rm SL}(N + 1;\mathbb C)\cdot x \bigcap \mathbb P_1\bigcap U$ is a $G$ orbit.

  \end{lem}

  \begin{proof}
  Since $G$ is reductive, we can decompose $V$ into $V=\mathbb C v_0\oplus V_1$ as a representation space of $G$,  where  $V_1$ can be identified with $T_{[v_0]} \mathbb P(V)$. Note that  $[v_0]$ is fixed by $G$. Then $G$ induces an action on $T_{[v_0]}$ which is the same as the action on $V_1$. For any $x\in \mathbb P(V)$, the map
  ${\rm SL}(N+1;\mathbb C)\rightarrow \mathbb P(V)$ by
  $$ {\rm SL}(N+1;\mathbb C)\rightarrow g\cdot x$$
   induces the tangential linear map $\iota_x: sl(N+1;\mathbb C)\rightarrow T_x \mathbb P(V)$.
    At $[v_0]$, we have  ${\rm Ker}(\iota_{[v_0]})=\mathfrak g$, where $\mathfrak g$ is the Lie algebra of $G$.

  Consider the representation of $G$ on the Lie algebra  ${\rm sl}(N+1;\mathbb C)$ of ${\rm SL}(N+1;\mathbb C)$.  Since $G$ is reductive, we have
  $${\rm sl}(N+1;\mathbb C)=g\oplus \mathfrak p,$$
   where $\mathfrak p$ is another representation of $G$. By the fact  ${\rm Ker}(\iota_{[v_0]})=\mathfrak g$, we get $p\cong \iota_{[v_0]}(p)$. Then we have a  decomposition   $T_{[v_0]} \mathbb P(V)=\iota_{[v_0]}(\mathfrak p)\oplus W$ for some subspace $W\subset V_1$. Thus
    $$V=\mathbb C v_0\oplus \iota_{[v_0]}(\mathfrak p)\oplus W$$
    and  $W\subset V$ is an invariant subspace as the representation $G$. For simplicity, we set
    \begin{align}\label{P1}\mathbb P_1=\mathbb P(\mathbb C v_0\oplus W).
    \end{align}

  Let   $f:\mathfrak p\times V\rightarrow \iota_{[v_0]}(\mathfrak p)$ be a map given by
  $$f(X,x)={\rm pr}(\exp(X)\cdot x),$$
    where ${\rm pr}$ means the projection onto ${\iota_{[v_0]}(\mathfrak p)}$. Clearly, $f(0, v_0)=0$ and the tangential map $\frac{\partial f}{\partial X}$ at $(0, v_0)$ is $id$. Then by  the  implicit function theorem, there is a neighborhood $\tilde U_1$ in $V$ of $v_0$ such that $\forall x\in \tilde U_1 $, we have ${\rm pr}(\sigma\circ x)=0$ for some $\sigma=\exp (X)$ near the identity. This is equivalent to that
    $$\sigma\cdot x\in \mathbb C v_0\oplus W.$$
     Denoting the image of $\tilde U_1$ in $\mathbb P(V)$ by $U_1$, 1) in Lemma \ref{slice}  is proved by taking $U_1$ as the neighborhood of $[v_0]$ in $\mathbb P(V)$.

     In order to prove 2),  we shrink $U_1$  a little bit. Since $\iota_{[v_0]}(\mathfrak p)$ is transversal to $T_{v_0}(\mathbb P_1)=W$, we can choose $U_2\subset U_1$ such that $\forall x\in U_2$, $\iota_x(\mathfrak p)$ is transversal to $T_x(\mathbb P_1)$:
  \begin{align}\label{trans}
  T_x\mathbb P(V)=\iota_x(\mathfrak p)\oplus T_x(\mathbb P_1) \text{ as vector spaces.}
  \end{align}
  By 1),  we can also choose $U\subset U_2$ such that $\forall x\in U$ there is some $\sigma\in {\rm SL}(N+1;\mathbb C)$ satisfying $\sigma\cdot x \in \mathbb P_1\bigcap U_2.$ Thus for any $x\in U$, we know that
    \begin{align}\label{g-case}G\cdot( \sigma\cdot x) \subset \mathbb P_1.
    \end{align}

  Let
   $$\gamma_t \subset {\rm SL}(N + 1;\mathbb C)\cdot x \bigcap \mathbb P_1\bigcap U,~t\in(-\epsilon,\epsilon)$$
   be a  path  with  $\gamma_0=\tau\circ x\in  \mathbb P_1\bigcap U$ for some $\tau\in {\rm SL}(N+1;\mathbb C)$. Then we can write $\gamma_t=\exp(X_t)\tau\cdot x$ for $t\in(-\delta,\delta)$ with small $\delta$. We need to show  that $X_t\in \mathfrak g$ and so  $\gamma_t\subset G(\tau\cdot x)$, which finishes the proof of 2). To prove $X_t\in \mathfrak g$, we only need to consider the case of $t=0$.  In fact, it is easy to see that
   $$\gamma'(0)=\iota_{\tau \cdot x}(X_0)\in T_{\tau \cdot x}(\mathbb P_1).$$
   Note that $\tau\cdot x \in U\subset U_2$.  Thus  by (\ref{trans}) and the fact ({\ref{g-case}}), we conclude  that $X_0\in \mathfrak g$. The lemma is proved.
  \end{proof}

  \begin{rem}
  In general, $SL(N + 1;\mathbb C)\cdot x \bigcap \mathbb P_1\bigcap U$ may consists of discrete points. In this case, if $\tau \cdot x \in \mathbb P_1\bigcap U$, then $\tau \cdot x$ is fixed by $G$.
  \end{rem}

  The following is elementary fact.

  \begin{lem}\label{orb}
  Let  $G$ be a reductive group and $V$  a representation space of $G$. Assume  that  $0\in \overline{G\cdot x}$ and $y\in \overline{G\cdot x}$ for some $x,y\in V$. Then $0\in \overline{G\cdot y}$.
  \end{lem}

  \begin{proof}
  On a contrary,  if $0\notin \overline{G\cdot  y}$,  there is a $G$-invariant polynomial $f$ such that $f(y)=1$. Since $y\in \overline{G\cdot x}$, we get $f(x)=1$.  But by the condition $0\in \overline{G\cdot x}$,  we  also have $f(x)=0$.  This is a contradiction!
  The lemma is proved.
  \end{proof}

 Combining Lemma \ref{slice} and Lemma \ref{orb}, we prove

   \begin{prop}\label{uniqueness-algebra } Let $x\in \mathbb CP^N$ and $x_0\in \overline{{\rm SL}(N + 1;\mathbb C)\cdot x}\subset \mathbb CP^N$. Suppose that the identity component of the stabilizer  $G\subset {\rm SL}(N + 1;\mathbb C) $  of $x_0$ is reductive. Then there is a neighborhood $U \subset  \mathbb CP^N$ of $x_0$ such that for any $x'\in  U\cap \overline{{\rm SL}(N + 1;\mathbb C)\cdot x}$ there are $\sigma\in {\rm SL}(N + 1;\mathbb C) $ and  a 1-${\rm PS}$ $\lambda(t)\subset  G$ such that
   \begin{align}\label{v0-conjugate}
   \lim_{t\rightarrow 0}\lambda(t) (\sigma \cdot x') = x_0.
   \end{align}

   \end{prop}

  \begin{proof}
   By Lemma \ref{slice},   there are  a $G$-invariant subspace $\mathbb P_1$ containing $x_0$ and a neighborhood $U$ of $x_0$ in $\mathbb CP^N$ such that the following holds:

   1) $\forall x \in U$, ${\rm SL}(N + 1;\mathbb C)\cdot  x$ intersects with $\mathbb P_1;$

  2) every component of ${\rm SL}(N + 1;\mathbb C)\cdot  x \bigcap \mathbb P_1\bigcap U$ is a $G$-orbit.
 \newline  Decompose $\overline {SL(N+1;\mathbb C)\cdot x} \bigcap \mathbb P_1$ into irreducible component $\bigcup_{i=1}^k C_i$. Assume that among all the irreducible components $C_i(1\leq i\leq l)$ contains $x_0$. Then we can replace $U$ by $U\setminus (\bigcup_{i=l+1}^k C_i)$, so that the closure of every component of ${\rm SL}(N + 1;\mathbb C)\cdot  x \bigcap \mathbb P_1\bigcap U$ contains $x_0$. Namely, for any $g\cdot x\in \mathbb P_1\bigcap U$, we have
  \begin{align}\label{clos}
  x_0\in \overline{ G(g\cdot x)}.
  \end{align}
 Thus  by  1),   for any $x'\in  U\cap \overline{{\rm SL}(N + 1;\mathbb C)\cdot x}$ there is $g_1\in {\rm SL}(N+1;\mathbb C)$ such that $g_1\cdot x'\in \mathbb P_1\bigcap U$. On the other hand,  since $g_1\cdot x'\in \overline{{\rm SL}(N+1;\mathbb C)\cdot x}$, there is a continuous family $\sigma_t \in SL(N+1;\mathbb C) (t\in (0,1])$ such that
 $$\lim_{t\rightarrow 0}\sigma_t\cdot x=g_1\cdot x'.$$
   We can assume that $\sigma_t\cdot x \in U$.   Hence,  by the proof of 1) in Lemma \ref{slice}, there is a continuous family $\tau_t\in SL(N+1;\mathbb C) (t\in (0,1])$ such that $(\tau_t\circ\sigma_t)\cdot x\in \mathbb P_1\bigcap U$ and $\lim_{t\rightarrow 0}\tau_t=id$.

  By  2) above,  $(\tau_t\circ\sigma_t)\cdot x ~ (t\in (0,1])$ is contained in a $G$-orbit. Namely
  there is a   $g_2\in  {\rm SL}(N+1;\mathbb C)$ such that
   ${G( g_2\cdot x)}\bigcap U$ contains
   $(\tau_t\circ\sigma_t)\cdot x ~(t\in (0,1])$.  As a conclusion,
   $$g_1 \cdot x'\in \overline{G( g_2\cdot x)}\bigcap U.$$
    Hence,   both of  $g_1 \cdot x'$ and $x_0$ lie in $\overline {G(g_2\cdot x)}$ by (\ref{clos}).

   Write $\mathbb P_1 = \mathbb P(x_0\oplus W)$ for a G-invariant subspace $V=x_0\oplus W$ as in (\ref{P1}). Regard three  points   $x_0$,  $g_2 \cdot x$ and  $g_1\cdot x'$ as vectors   $0$,   $x$ and  $y$ in $V$ as  Lemma \ref{orb},  respectively.  Then  we have $0 \in \overline{G \cdot y}$.  Thus by Luna's lemma,   there is a 1-PS $\lambda(t)$ such that
   $$\lim_{t\rightarrow 0}\lambda(t)(g_1 \cdot x') = x_0.$$

  \end{proof}

   \subsection{Uniqueness of algebraic structure $\tilde M_\infty$}

 Let $M\subset \mathbb CP^N$ be a Fano manifold embedded   in  $\mathbb CP^N$. Let  $\mathcal C_{KE}$ be a set of $Q$-Fano varieties   which  consists of all possible  limits  under ${\rm SL}(N+1, \mathbb C)$-group on $M$   with $klt$-singularities   and admitting a singular KE metric.

  \begin{prop}\label{dis}
 Given $\tilde M_\infty \in \mathcal  C_{KE}$, there exists an $\epsilon>0$ such that  $\tilde M_\infty=g\cdot \tilde M'_\infty$ for some $g\in {\rm SL}(N+1;\mathbb C)$, if  $\tilde M'_\infty\in \mathcal C_{KE}$ satisfies
  $$d([\tilde M_\infty],[\tilde M'_\infty]) \leq \epsilon.$$
  \end{prop}

   \begin{proof}We prove  the proposition by contradiction.
 Regard  $[  \tilde M_\infty]$ as a Chow point in $W_n=\mathbb  P(\mathbb C(M_{(n+1)\times (N+1)}))$ as in Section 2. Let
     $$\mathcal [C_{KE}]=\{[\tilde  M_\infty]|~  \tilde  M_\infty\in \mathcal C_{KE}\}.$$
 On the contrary,   there is a sequence of $[\tilde M_i] \in [\mathcal C_{KE}]$  which converges to  $[\tilde M_\infty]$, but
   \begin{align}\label{equival} [\tilde M_i]\neq  g\cdot[\tilde M_j], ~\forall ~g\in {\rm SL}(N + 1;\mathbb C),~i\neq j.
   \end{align}
  By Proposition \ref{red} in Appendix, we know that $G={\rm Aut}_0(M_\infty)\subset  {\rm SL}(N + 1;\mathbb C)$ is reductive. Thus applying Proposition \ref{uniqueness-algebra } to  $[M_\infty]$ and $[\tilde M_i]$ as $i>>1$,  there are $\sigma_i\in {\rm SL}(N + 1;\mathbb C) $ and  a 1-${\rm PS}$ $\lambda(t)\subset  G$ such that
   \begin{align}\label{v0-conjugate}
   \lim_{t\rightarrow 0}\lambda(t) (\sigma_i \cdot [\tilde M_i]) = [\tilde M_\infty].
   \end{align}
   However, $\tilde M_i$ is K-polystable \cite{Ber},   it must hold that $\lambda(t)$ preserves $\sigma_i\cdot [\tilde M_i]$.  Thus $\sigma_i\cdot  [\tilde M_i]=[\tilde M_\infty]$. As a consequence, $[\tilde M_i]=(\sigma_i^{-1}\cdot\sigma_j)\cdot[\tilde M_j]$. This is a contradiction with (\ref{equival}).

   \end{proof}

   \begin{cor}\label{semiuni-KE}The set  $\mathcal C_{KE}/{\rm SL}(N+1, \mathbb C) $ is finite.
   \end{cor}

 \subsection{In case of reductive  ${\rm Aut}_0(\tilde M_\infty)$}

  For a version  of KR solitons in  Proposition \ref{dis}, we introduce a set $\mathcal C_{KS}(X)$ of $Q$-Fano varieties  which  consists of all possible  limits  $\tilde M_\infty$ under ${\rm SL}(N+1, \mathbb C)$-group on $M$   with $klt$-singularities  and admitting  a singular KR soliton with respect to a same  holomorphic VF   $X\in {\rm sl}N+1;\mathbb C)$.  In this subsection, we assume that  $~{\rm Aut}(\tilde M_\infty)$ is reductive.  Namely, we set
  $$\mathcal C_{KS}^0(X)=\{\tilde M_\infty\in\mathcal C_{KS}(X)|~{\rm Aut}_0(\tilde M_\infty) ~{\rm is ~reductive}\}. $$
   Then we  have

  \begin{prop}\label{dis2}
 Given $\tilde M_\infty \in \mathcal  C^0_{KS}(X)$, there exists an  $\epsilon>0$ such that $\tilde M_\infty=g\cdot \tilde M'_\infty$ for some $g\in {\rm SL}(N+1;\mathbb C)$, if $\tilde M'_\infty\in \mathcal C_{KS}(X)$ satisfies
   $$d([\tilde M_\infty],[\tilde M'_\infty])\leq \epsilon.$$
  \end{prop}

   \begin{proof} The proof is also same to one of  Proposition \ref{dis}. In fact, if Proposition \ref{dis2}  is not true, then
   there are  a $\tilde  M_\infty\in \mathcal C_{KS}^0(X)$ and  a sequence  of $[\tilde M_i] \in [\mathcal C_{KE}]$ such that $[\tilde M_i]$ converges to  $[\tilde M_\infty]$ such that
    (\ref{equival}) holds.
  Since $G={\rm Aut}_0(\tilde M_\infty)\subset  {\rm SL}(N + 1;\mathbb C)$ as a stabilizer of $[M_\infty]$  is reductive,
   applying Proposition \ref{uniqueness-algebra } to  $[\tilde M_\infty]$ and $[\tilde M_i]$ as $i>>1$,  there are $\sigma_i\in {\rm SL}(N + 1;\mathbb C) $ and  a 1-${\rm PS}$ $\lambda(t)=\exp\{t\eta\}\subset  G$ such that
  (\ref {v0-conjugate}) also holds, where $\eta$ is an element of Lie algebra of $G$. By Proposition \ref{center} in Appendix, the soliton VF $X$ on $\tilde M_\infty$ lies  in the center of Lie algebra $G$. Thus  $\lambda(t)$ communicate with $X$.  We claim
    \begin{align}\label{lambda}\lambda(t)\subset {\rm Aut}(\sigma_i\circ M_i).
      \end{align}
      Thus  $\lambda(t)$ preserves $\sigma_i\cdot [\tilde M_i]$.  Thus $\sigma_i\cdot  [\tilde M_i]=[\tilde M_\infty]$. As a consequence, $[\tilde M_i]=(\sigma_i^{-1}\cdot\sigma_j)\cdot[\tilde M_j]$. This is a contradiction with (\ref{equival}).

   Suppose that $\lambda(t)\subset {\rm Aut}(M_i)$ dosen't hold.  Then by the relative K-polystability of  $\tilde M_i$ \cite{BW, DatS},  the  modified Ding-Tian invariant $F_X(\eta)$ (also called modified  Futaki invariant) with respect to   $\lambda(t)$ is positive.  However,  $F_X(\eta)=0$  since $\tilde M_\infty$ admits  a KR soliton. Thus, we get a contradiction, and  (\ref{lambda}) must be true. The proposition is proved.

   \end{proof}

   \begin{cor}\label{semiuni-KR-solions}
     The set  $\mathcal C_{KS}^0(X)/{\rm SL}(N+1, \mathbb C) $ is finite.
   \end{cor}

 \begin{theo}\label{unique-algebra}Let $(\tilde M_\infty, \hat \omega_\infty)$ be  a singular KR soliton limit of  a sequence  $\{(M, \omega_{t_i})\}$ of  (\ref{kr-flow})  as in Theorem \ref{WZ}. Suppose that $ {\rm Aut}_0(\tilde M_\infty)$ is reductive. Then for any sequence  $\{(M, \omega_{t_i'})\}$,
 there is a subsequence of $\{(M, \omega_{t_i'})\}$ which  is locally  $C^\infty$-convergent to  $\hat \omega_\infty$ on  ${\rm Reg}(\tilde M_\infty)$ in the Cheeger-Gromov topology. In particular, if the limit $(\tilde M_\infty, \hat\omega_\infty)$ is  a singular  KE metric, then for any sequence  $\{(M, \omega_{t_i'})\}$,
 there is a subsequence of $\{(M, \omega_{t_i'})\}$ which  is locally  $C^\infty$-convergent to  $\hat \omega_\infty$ on  ${\rm Reg}(\tilde M_\infty)$ in the Cheeger-Gromov topology. As a consequence,  $(M, \omega_{t})$  converges  uniformly to
 $\overline{(\tilde M_\infty, \hat \omega_\infty)}$ in the
  Gromov-Hausdroff topology.
 \end{theo}

  \begin{proof}
 By assumption we have $\tilde M_\infty\in \mathcal C_{KS}^0(X)$ for some holomorphic VF $X$.  By Corollary \ref{vec} and Theorem \ref{WZ}, we know that $\mathcal C_0\subseteq \mathcal C_{KS}(X)$ for the same $X$. Thus it suffices to show that $\mathcal C_0={\rm SL}(N+1;\mathbb C)\cdot [\tilde M_\infty]$ since the convergence of  (\ref{kr-flow}) depends only on  $\tilde M_\infty$  according to the proof of Theorem \ref{WZ} by using the uniqueness of singular KR solitons  on a Q-Fano variety with klt-singularities \cite{Bern}.  On the contrary, if $\mathcal C_0/{\rm SL}(N+1;\mathbb C)$ is not a single point, then there is another point $\tilde M'_\infty$ such that ${\rm dist}([\tilde M'_\infty], {\rm SL}(N+1;\mathbb C)\cdot [\tilde M_\infty])
 \geq \delta >0$ by Corollary \ref{semiuni-KR-solions}. Because $\mathcal C_0$ is connected, there is a point $[\tilde M''_\infty]$ such that
  ${\rm dist}([\tilde M^{''}_\infty], {\rm SL}(N+1;\mathbb C)\cdot [\tilde M_\infty])=\frac{\epsilon}{2}$, where  $\epsilon$  is  the constant determined  in Proposition \ref{dis2}.  However,  it is impossible by  Proposition \ref{dis2}  and so  the theorem  is proved.
 \end{proof}

 In Section 5, we will see that the reductivity condition of  $ {\rm Aut}_0(\tilde M_\infty)$ in  Theorem \ref{unique-algebra} can be removed as shown in \cite{CSW}.

 \section{Limits of singular KE metrics }

  In this section, we prove Theorem \ref{general-KR}  in case that $(\tilde M_\infty, \hat \omega_\infty)$ is a singular KE metric. The idea is to deform   KR flows (\ref{kr-flow}) with varied initial metrics   as done for the smooth convergence in \cite{TZ4, TZZZ}.  For any $\omega_0'\in 2\pi c_1(M,J)$,   we let $\omega^s=s\omega_0+(1-s)\omega_0'$ ($s\in [0,1]$). We want to prove the global convergence of  flow  $(M, \omega_t^s)$ for any initial $\omega_s$. First we prove   the following stability  result of KR flow  $(M, \omega_t')$ when its initial metric $\omega_0'\in 2\pi c_1(M,J)$ is very closed to   $\omega_0$.

 \begin{theo}\label{stability-KE} Suppose that the limit $(\tilde M_\infty, \hat\omega_\infty)$ in Theorem \ref{WZ} is  a singular  KE metric. Then there is an $\epsilon>0$ such that for any initial metric  $\omega_0'\in 2\pi c_1(M,J)$ with
 \begin{align}\label{small-initial-0}\|\omega_0'-\omega_0\|_{C_{CG}^2(M)}\le \epsilon,
 \end{align}
  flow $(M, \omega_t')$  is globally  convergent to a  Gromov-Hausdroff limit   $(M_\infty, \omega_\infty)$,  which  is
    the completion of $({\rm Reg}(\tilde M_\infty), \hat\omega_\infty)$.
 \end{theo}

  \begin{proof} Let $(M_\infty', \omega_\infty')$ be a Gromov-Hausdroff limit of any sequence $\{\omega_{t_i}'\}$ of $\omega_{t}'$.  We claim:  For any $\delta>0$ there is an $\epsilon>0$ such that
   \begin{align}\label{small-GH-3}{\rm dist}_{GH}(  (M_\infty, \omega_\infty),    (M_\infty', \omega_\infty'))\le \delta,
   \end{align}
  if  $\omega_0'$ satisfies (\ref{small-initial-0}).
  Then  $\tilde M_\infty$ must be biholomorphic to $\tilde M_\infty'$ by Proposition \ref{dis2} together with Proposition \ref{two-topology}, where  $\tilde M_\infty'$ is the Q-Fano variety associated to $(M_\infty', \omega_\infty')$.  The theorem is proved.

  We prove (\ref{small-GH-3})  by contradiction. On the contrary,  there are a number $\delta_0>0$ and   a sequence of  $\omega_0^{\alpha_i}\in 2\pi c_1(M,J)$ such that
  $$\|\omega_0^{\alpha_i}-\omega_0\|_{C_{CG}^2(M)}\to 0,~{\rm as}~\alpha_i\to\infty,$$
  and
 \begin{align}\label{contradiction-3}  {\rm dist}_{GH}(  (M_\infty, \omega_\infty),    (M_\infty', \omega_\infty'))\ge \delta_0,
 \end{align}
 where
 $ (M_\infty', \omega_\infty')$ is a  Gromov-Hausdroff limit of some  sequence $\{\omega_{t_k}^{\alpha_i}\}$ in  (\ref{kr-flow}) with an initial metric $\omega_0^{\alpha_i}$.
 Since $(M, \omega_t)$ is globally  convergent to $ (M_\infty, \omega_\infty)$ in Gromov-Hausdroff by  Theorem \ref{unique-algebra}  by (\ref{contradiction-3}), we can also choose a sequence   $\{\omega_{t_i}^{\alpha_i}\}$ which converges to  $(\bar M_\infty, \bar \omega_\infty))$
 in  Gromov-Hausdroff topology such that
 \begin{align}\label{contradiction-2}  {\rm dist}_{GH}(  (M_\infty, \omega_\infty),    (\bar M_\infty, \bar \omega_\infty))= \frac{\delta}{2},
 \end{align}
 where $\delta(=\epsilon)$ is determined as in (\ref{small-GH-0}) in  Proposition \ref{two-topology}.
 Moreover, by the monotonicity  of Perelman's entropy \cite{Pe, TZ4, WZ20},
  we have
  \begin{align}\label{max-entropy-comutation}&\lim_{t\rightarrow\infty, \alpha_i\to 0}\lambda(\omega_{t_i}^{\alpha_i})\ge L(\omega_0)=\lim_{t_i\rightarrow\infty}\lambda(\omega_{t_i}) \notag\\
 & =\sup\{\lambda(g')|~\omega_{g'}\in 2\pi c_1(M,J)\}=
   \int_M  c_1(M,J)^n.
   \end{align}
   It follows that
   $$\lim_{t\rightarrow\infty, \alpha_i\to 0}\lambda(\omega_{t_i}^{\alpha_i})=   \int_M  c_1(M,J)^n.$$
  Thus as in the proof of Theorem \ref{WZ} (also see \cite[Lemma 4.2]{WZ20}),
  we  can prove that $\omega_{t_i}^{\alpha_i}$ is also locally  $C^\infty$-convergent to a  singular KE metric on a  Q-Fano variety  with klt singularities. Hence,  by  Proposition \ref{dis2}  together with Proposition   \ref{two-topology}, we conclude that  $(M_\infty, \omega_\infty)$ is isometric to    $(\bar M_\infty, \bar \omega_\infty)$ as in the  proof of  Theorem \ref{unique-algebra}.   This is impossible by (\ref{contradiction-2}). Therefore, (\ref{small-GH-3}) is true and  Theorem \ref {stability-KE} is proved.
  \end{proof}

  By Theorem \ref{stability-KE}, we set
  \begin{align}\label{i-set}I=\sup_s\{[0,s]| ~(M, \omega_t^\tau)~     \stackrel{GH}{\longrightarrow } (M_\infty, \omega_\infty)~{\rm uniformly}, ~\forall~\tau\in [0,s]\}.
  \end{align}
 Note  that $0\in I$ by Theorem \ref{unique-algebra} (also see Theorem   \ref{mainThm-KS} below). Then  $I$ is an open set.  Thus to show $I=[0,1]$,  it  remains to prove that $I$ is also closed.  Without of
  loss of generality, we may assume that $I=[0, s_0)$.
   This means that $\omega_t^s$ is uniformly   convergent to the Gromov-Hausdroff limit   $(M_\infty, \omega_\infty)$ for any $s\le s_0$,  which  is
    the completion of $({\rm Reg}(\tilde M_\infty), \hat\omega_\infty)$.

 Recall  the energy level $L(\cdot )$ of Perelman entropy
  $\lambda(\cdot)$ for   flow  $(M,\omega_t')$  with  the initial metric $\omega'_0$  defined by  (cf. \cite{TZ4}),
  \begin{equation}\label{level}
  L(\omega_0')=\lim_{t\rightarrow\infty}\lambda(\omega_t).
  \end{equation}
  To show $s_0\in I$,  we need to prove that
  $$L(\omega^{s_0})=L(\omega_0).$$
 By a result in  \cite{TZ3},  it suffices to get a lower bound of K-energy in the class of $2\pi c_1(M,J)$.  In the following,  we will verify it by showing Ding-energy bounded below as done by Li to solve certain homogeneous complex Monge-Amp\`ere equation  \cite{Li} .

 \subsection{Lower bound of Ding-energy}

  For a smooth K\"ahler metric $\omega\in 2\pi c_1(M)$, choose a Hermitian metric $h$ on $K_M^{-1}$ such that $R(h)=\omega$. $h$ is the same as a volume form $dV_h$ and its Ricci curvature is $\omega$.  For  any $\phi\in {\rm PSH}(M,\omega)$, we define
  $$F^0_\omega(\phi)=-\frac{1}{n+1} \frac{1}{V}\sum_{i=0}^n \int_M \phi\omega^{n-i}\wedge \omega_\phi^{i}$$ and
  $${\rm Ding}_\omega(\phi)=F^0_\omega(\phi)-\log\int_M e^{-\phi}dV_h.$$

  The following is due to a result of Li  \cite{Li}.

  \begin{prop}\label{low}
  Let $\pi: \mathcal  X\rightarrow \Delta$ be a special degeneration  of Fano manifold $M$ such that the central fiber admitting a singular  KE metric. Then we have $ {\rm Ding}_\omega(\phi)\geq C$.
  \end{prop}

  \begin{proof}
  Fix an admissible metric $\Omega$ on $\mathcal X$ and $\Omega|_{\pi^{-1}(1)}=\omega$. Choose a Hermitian metric $h_{\mathcal X}$ on $K^{-1}_{\mathcal X}$ such that $R(h_{\mathcal X})=\Omega$. For any $\phi\in {\rm PSH}(M,\omega)$, we can solve
  $$(\Omega+dd^c \Phi)^{n+1}=0 \text{ on }\mathcal  X,$$
  with boundary value $\Phi_{\pi^{-1}(1)}=\phi$ (cf. \cite{Ko98, EGZ}).   Then $\Phi$ is bounded  on $\Delta$ with $S^1$-invariance.
 Now we define a function on $\Delta$ by
  $$f(t)={\rm Ding}_{\Omega|_{t}}(\Phi|_{t})=F^0_{\Omega|_{t}}(\Phi|_{t})-\log\int_{\mathcal X_t} e^{-\Phi}dV_{h_{\mathcal X}}.$$
   By Lemma \ref{con} below, $f(t)$ is continuous.  We will show that $f(t)$ is a subharmonic function of $t$.  Since the second term  of $f(t)$ is subharmonic by Berndtsson's result \cite{Bern},  we need to show that  $F^0_{\Omega|_{t}}(\Phi|_{t})$  is also a  subharmonic function.
   Similar to \cite{Bern}, we can use the method of integral by part.
 In fact, by choosing a non-negative test function $\psi(t)$ with zero boundary on $\Delta$, we get
  \begin{align}
  \int_\Delta \psi\sqrt{-1}\partial \bar\partial F^0_{\Omega|_{t}}(\Phi|_{t})&=\int_\Delta \sqrt{-1}\partial \bar\partial \psi
  (-\frac{1}{n+1} \frac{1}{V}\sum_{i=0}^n \int_{\mathcal X_t} \Phi_t \Omega_t^{n-i}\wedge \Omega_\Phi^{i}) \notag \\
  &=(-\frac{1}{n+1})\frac{1}{V}\int_{\mathcal X}\sqrt{-1}\partial \bar\partial \psi \sum_{i=0}^n \Phi \Omega^{n-i}\wedge \Omega_\Phi^{i})\notag\\
  &=(-\frac{1}{n+1})\frac{1}{V}\int_{\mathcal X}\psi \sum_{i=0}^n (\sqrt{-1}\partial \bar\partial \Phi) \Omega^{n-i}\wedge \Omega_\Phi^{i})\notag\\
  &=(-\frac{1}{n+1})\frac{1}{V}\int_{\mathcal X}\psi(\Omega_\Phi^{n+1}-\Omega^{n+1})=\frac{1}{n+1}\frac{1}{V}\int_{\mathcal X}\psi\Omega^{n+1}\geq 0\notag.
  \end{align}
   Thus $f(t)$ is subharmonic. By the maximum  principle, $f(0)\leq f|_{\partial \Delta}=f(1)$. We note that $f(0)$ is the Ding energy of $\mathcal X_0$ with respect to $\Omega_0$, which  is  bounded  below by \cite{BBEGZ}. Hence, $f(t)\ge -C$ for some constant $C$.  The proposition is proved.
  \end{proof}

  \begin{lem}\label{con}
  Both functions $F^0_{\Omega|_{t}}(\Phi|_{t})$ and  $\int_{\mathcal X_t} e^{-\Phi}dV_{h_{\mathcal X}}$ are continuous function of $t$.
  \end{lem}

  \begin{proof}
  We will divide the integral into two parts for each function in the lemma: near singularities and away from singularities.  Denote the integrand of   $F^0_{\Omega|_{t}}(\Phi|_{t})$  by
   $$\mathcal R=-\frac{1}{n+1} \frac{1}{V}\sum_{i=0}^n \int_{\mathcal X_t} \Phi_t \Omega_t^{n-i}\wedge \Omega_\Phi^{i}.$$
  Since the central fiber $\mathcal X_0$ is normal,  $\mathcal {\rm Sing}(\mathcal X_0) $ is pluripolar. Thus we can choose neighborhoods $W_\delta$ of $ {\rm Sing}( \mathcal X)=\mathcal {\rm Sing}(\mathcal X_0)$ such that
  \begin{align}\label{small}
  \lim_{\delta\rightarrow 0} \int_{\mathcal X_0\bigcap W_\delta} \Omega_t^{n-i}\wedge \Omega_\Phi^{i}=0, i=1,2,...,n.
  \end{align}
  It follows that
  \begin{align}\label{f-f0}&|F^0_{\Omega|_{t}}(\Phi|_{t})-F^0_{\Omega|_{0}}(\Phi|_{0})|\notag\\
  &\leq |\int_{\mathcal X_t\setminus W_\delta}\mathcal R-\int_{\mathcal X_0\setminus W_\delta}\mathcal R|+|\int_{W_\delta\bigcap \mathcal X_t}\mathcal R|+|\int_{W_\delta\bigcap \mathcal X_0}\mathcal R|.
   \end{align}

  The first term in (\ref{f-f0}) converges to 0 by the $C^{1,\alpha}$ regularity of $\Phi$ in ${\rm Reg}(\mathcal X)$ and  the continuity of Monge-Amp\'ere mass. For the second term, we have
  $$|\int_{\mathcal X_t\setminus W_\delta}\mathcal R|\leq |\Phi|_{L^\infty}\frac{1}{(n+1)V}\sum_{i=0}^n \int_{\mathcal X_t\setminus W_\delta}\Omega_t^{n-i}\wedge \Omega_\Phi^{i}$$
    and
  \begin{align}
  \frac{1}{(n+1)V}\sum_{i=0}^n \int_{\mathcal X_t\setminus W_\delta}\Omega_t^{n-i}\wedge \Omega_\Phi^{i}&=1-\frac{1}{(n+1)V}\sum_{i=0}^n\int_{\mathcal X_t\bigcap W_\delta}\Omega_t^{n-i}\wedge \Omega_\Phi^{i}\notag\\
  &\rightarrow 1-\frac{1}{(n+1)V}\sum_{i=0}^n\int_{\mathcal X_0\bigcap W_\delta}\Omega_t^{n-i}\wedge \Omega_\Phi^{i}( t\rightarrow 0)\notag\\
  &=\int_{\mathcal X_0\bigcap W_\delta}\Omega_t^{n-i}\wedge \Omega_\Phi^{i}.\notag
  \end{align}
  Thus  together with (\ref{small}), all the three terms in (\ref{f-f0}) can be arbitrarily small when $\delta,t$ are small enough.

  For the continuity of $\int_{\mathcal X_t} e^{-\Phi}dV_{h_{\mathcal X}}$, we will construct the suitable neighborhoods $W_\delta$ such that
  \begin{align}\label{volumesmall}
  \lim_{t\rightarrow} \int_{\mathcal X_t\bigcap W_\delta}dV_{h_{\mathcal X}}=\int_{\mathcal X_t\bigcap W_\delta} dV_{h_{\mathcal X}}.
  \end{align}
  Then we can prove the continuity of $\int_{\mathcal X_t} e^{-\Phi}dV_{h_{\mathcal X}}$ in the same way as above for $F^0_{\Omega}$.

  Choose a  resolution of singularity of
  $(\mathcal X,\mathcal X_0)$ by
  $$\mu: \tilde {\mathcal X}\rightarrow \mathcal X$$
  such that $\mu^{-1}\mathcal X_0$ is a SNC divisor on $\tilde{\mathcal X}$. Write
  $$\mu^{-1}\mathcal X_0=\mathcal X'_0+\sum a_i E_i, $$
    where $\mathcal X'_0$ is the proper inverse image of $\mathcal X_0$. From the definition of discrepancy and adjunction formula, we have
    $$K _{\tilde{\mathcal X}/\mathbb C}+\mathcal X'_0=\mu^*(K_{\mathcal X}/\mathbb C+\mathcal{X}_0)-\sum_{i=1}^l b_iE_i( 0\leq b_i<1)$$
     and
     $$K _{{\mathcal X_0}'}=\mu^* K_{\mathcal X_0}-\sum_{i=1}^l b_iE_i.$$
     %  Now we can construct the neighborhoods on $\tilde{\mathcal X}$ locally.
     For a point $\tilde x\in \mathcal X_0'\bigcap_{i=1}^k E_i(k\leq l)$, we can choose a local coordinates
     $(w_0,w_i(1\leq i\leq n))$ such that
      $$\mathcal X_0'=\{w_0=0\},~E_i=\{w_i=0\}$$ and the map
      $\tilde {\mathcal X}\rightarrow \mathbb C$ is given by
      $$t=w_0\Pi_{i=1}^l w_i^{b_i}. $$
      Then  we define a neighborhood around $\tilde x$ :
  \[
  \textstyle{
  \tilde{U}_t(\tilde{x},\delta)=\left\{\left(\frac{t}{\prod_{i=1}^{l}w_i^{a_i}},w'\right)\in\!\mathbb{C}^{n+1};\! |w_j|\le \delta, j=1, \dots, n,\! \mbox{ and } \!\left|\frac{t}{\prod_{i=1}^{l}w_i^{a_i}}\right|\le \delta\right\},
  }
  \]
  where $w'=(w_1,...,w_n)$.
Thus,  under the above coordinates, for  a local non-vanishing section of $K_{\mathcal X}$, we can write it as
  $$\mu^*(v)=g(w)\Pi_{i=1}^l w^{a_i-b_i}(dw_0\wedge dw'\otimes dt)$$
  and
  $$(\mu|_{\mathcal X_0'})^* (v_0)=g(0,w')\Pi_{i=1}^{l} w_i^{-b_i} dw',$$ where $g$ is a non-vanishing holomorphic function and $v_0$ is the restriction of $v$ on $\mathcal X_0$.
 Hence,   the volume $dV_{h_{\mathcal X_t}}$ on $\mathcal X_t$ is given by
  $$\mu^*(v\wedge\bar{v})=|g(w_0(t,w'),w')|^2\bigwedge_{i=1}^{l}|w_i|^{-2b_i}dw_i\wedge d\bar{w}_i\wedge\bigwedge_{j=l+1}^n dw_j\wedge d\bar{w}_j.$$
  By the pointwise convergence, it follows that
  \begin{align}
  \lim_{t\rightarrow 0}\mu^*(v\wedge\bar{v})^{1/m}&=|g(0,w')|^2\bigwedge_{i=1}^{l}|w_i|^{-2b_i}dw_i\wedge d\bar{w}_i\wedge\bigwedge_{j=l+1}^{n}dw_j\wedge d\bar{w}_j\notag \\
  &=\mu|_{\mathcal X_0'}^* (v_0\wedge\bar{v}_0)^{1/m}\notag.
  \end{align}
  Since $g$ is bounded, we derive by  dominant convergence,
  $$\lim_{t\rightarrow 0}\int_{\tilde{{U}}_t(\tilde{x},\delta)}\mu^*(v\wedge\bar{v})=\int_{\mathcal X'_0\cap\tilde{U}(\tilde{x},\delta)}\mu|_{\mathcal X'_0}^*(v_0\wedge\bar{v}_0). $$  This is the same as (\ref{volumesmall}). For the points not in $\mathcal X'_0$, we can do a similar calculation. Therefore,  the continuous of $\int_{\mathcal X_t} e^{-\Phi}dV_{h_{\mathcal X}}$ is also proved.
  \end{proof}

 \subsection{Proof of Theorem \ref{general-KR}  in  case of singular KE metrics} In this subsection, we prove

 \begin{theo}\label{mainThm-KE}
    Suppose that there is a sequence  $(M, \omega_{t_i})$ of  (\ref{kr-flow}) such that $\omega_{t_i}$ converges to a  singular KE metric $\hat\omega_\infty$ as  in Theorem \ref{WZ}.  Then  $\omega_t$ is locally uniformly  convergent to  $\hat\omega_\infty$  on   ${\rm Reg}(\tilde M_\infty)$ in the Cheeger-Gromov topology.     As a consequence,   $(M, \omega_t)$ is uniformly  convergent to a  Gromov-Hausdroff limit   $(M_\infty, \omega_\infty)$, which  is
    the completion of $({\rm Reg}(\tilde M_\infty), \hat\omega_\infty)$  and  is independent of choice of initial metrics in  $2\pi c_1(M, J)$.

     \end{theo}

  \begin{proof}  By Proposition \ref{low}, we see that   for any  $\omega'\in 2\pi c_1(M,J)$, it holds \cite{TZ3},
   \begin{align}\label{KE-entropy} L(\omega')=\sup\{\lambda(g')|~\omega_{g'}\in 2\pi c_1(M,J)\}=
   \int_M  c_1(M,J)^n.
  \end{align}
  Applying Theorem  \ref{unique-algebra} to  the KR flow with the initial metric $\omega^{s_0}$ together with \cite[Proposition 4.14]{WZ20},   there is   a Q-Fano variety $\tilde M_\infty'$ with admitting  a weak KE metric   $\hat \omega_\infty'$  such that $\omega_t^{s_0}$ is locally  $C^\infty$-convergent to   $\hat \omega_\infty'$  on  ${\rm Reg}(\tilde M_\infty')$ in the Cheeger-Gromov topology and
  \begin{align}\label{two-iso-2}(M_\infty', \omega_\infty')=\overline{({\rm reg}(\tilde M_\infty'), \hat\omega_\infty')},
  \end{align}
 where $(M_\infty', \omega_\infty')$ is the global  Gromov-Hausdroff limit of $(M, \omega_t^{s_0})$. Then by Theorem \ref{stability-KE}, there is a small $\epsilon_0$ such that   flow  $(M, \omega_t^{s_0-\epsilon_0})$  is also globally convergent to $(M_\infty', \omega_\infty')$  in   Gromov-Hausdroff topology.  Note that $s_0-\epsilon_0\in I$. Thus
 \begin{align}\label{hg-isometric}
   (M_\infty', \omega_\infty')\cong  (M_\infty, \omega_\infty).
   \end{align}
   Moreover, as in the proof of (\ref{open-iso}), we can also get
  $$ ({\rm reg}(\tilde M_\infty'), \hat\omega_\infty')\cong ({\rm reg}(\tilde M_\infty), \hat\omega_\infty).$$
  Hence,  we also prove that $\omega_t^{s_0}$ is locally  $C^\infty$-convergent to   $\hat \omega_\infty$  on  ${\rm Reg}(\tilde M_\infty)$ in the Cheeger-Gromov topology.
   Theorem \ref{mainThm-KE} is proved.
  \end{proof}

 \section{Computation of $L(\omega_0')$ and applications }

 In this section,  we first give a brief sketch of the eigenvalues and eigenspaces  estimate for representation group in \cite{CSW}, then apply it to
  compute  the energy level $L(\omega')$ of flow (\ref{kr-flow}) through estimating $H$-invariant introduced in  \cite{TZ2, TZZZ, He}.
 We will represent a result of  Dervan-Sz\'ekelyhidi  that $L(\omega')$ is independent of $\omega'$  \cite{DS16}.  Here the proof depends
 on the partial $C^0$-estimate and local estimate of K\"ahler potentials in \cite{WZ20}.

  Let $E$ be a complex linear space and $G={\rm GL}(E), K=U(E)$. Assume that $V$ be a representation of $G$. Fix a $K$-invariant metric on $V$. Let $A_i(i=0,1,...)$ be a sequence of matrices in $G$ and $\Lambda$ be a Hermitian matrix. Assume that
  \begin{align}\label{p}
  \lim_{i\rightarrow \infty} A_iA_{i-1}^{-1}=e^\Lambda.
  \end{align}
  Denote the eigenvalues of $\Lambda$ by $\mathcal S=\{d_1>d_2>d_3...>d_{k-1}>d_k\}$ and the eigenspaces corresponding to $d_j$ by $U_j$. For $v\in V\setminus \{0\}$, $[v]$ is the corresponding point in the projective space $\mathbb P(V)$. Denote the limit points of $A_i\cdot [v]$ by ${\rm Lim}[v]$.

  The following two lemmas can be found in  \cite{CSW}.

  \begin{lem}\label{lim}
  Assume that  (\ref{p}) holds, then for $v\in V\setminus \{0\}$, $d(v)=\lim_{i\rightarrow \infty}(\log |A_{i+1} v|-\log |A_{i} v|)$ exists and belongs to $\mathcal S$. Moreover for any $[w]\in {\rm Lim}[v]$, $w$ belongs to $U_{d(v)}$.
  \end{lem}

  Define $V_{j}=\{v\in V| d(v)\leq d_j\}$. Then $\dim V_j=\sum_{i=j}^k \dim U_i.$ Moreover we have

  \begin{lem}\label{fil}
  Assume that $V=S^m(E)$ for some $m\geq 1$. Then there exists $C\in G$, which is independent of $m$ such that $C\cdot V_j=\oplus_{i=j}^k U_i$.  There also exists  $C_i\in G$ with $\lim_{i\rightarrow \infty} C_i=Id$ such that
  $\tilde A_i=C_i^{-1}A_iC \in G_{\Lambda}=\{g\in G|~g\cdot e^{s\Lambda}=e^{s\Lambda}\cdot g\}$.
  \end{lem}

  \begin{prop}\label{alg}
  Assume that $V=S^m(E)$ for some $m\geq 1$. Define $\bar v=\lim_{i\rightarrow \infty} e^{i\Lambda} C\cdot v$. Then ${\rm Lim}[v]\subset \overline{ G_\Lambda \cdot [\bar v]}$.
  \end{prop}

  \begin{proof}
  Assuming that $d(v)=d_j$, we have $v\in V_j$. So $C\cdot v\in \oplus_{i=1}^j U_i$, and the $U_j$ component $\pi_j(C\cdot v)$ of $C\cdot v$ is not zero. It follows that $[\bar v]=[\pi_j(C\cdot v)]$. For $[w]\in {\rm Lim}[v]$, by Lemma \ref{lim}, we know that $w\in U_j$. Assume that $[w]=\lim_{i\rightarrow \infty} [A_{\alpha_i}v]$, then we have
  \begin{align}
  \lim_{i\rightarrow \infty} \tilde A_{\alpha_i} [\bar v]&=
  \lim_{i\rightarrow \infty} \tilde A_{\alpha_i} [\pi_j(C\cdot v)]
  =\pi_j(\lim_{i\rightarrow \infty} \tilde A_{\alpha_i} [(C\cdot v)]) \notag \\
  &=\pi_j(\lim_{i\rightarrow \infty} A_{\alpha_i}[v])=\pi_j[w]=[w].  \notag
  \end{align}
  The proposition is proved.
  \end{proof}

 For the metric $\omega_t$ of K\"ahler-Ricci flow, we have an induced Hermitian metric $h_t$ on $E=H^0(M,K_M^{-l})$ of dimension $N+1$. Consider the solutions of the following ODE system:
  \begin{align}\label{ODE-s}
  \frac{d}{dt}s^\alpha(t)=-\frac{1}{2}\sum_{\gamma=0}^{N} H_t'(s^\alpha, s^\gamma)s^\gamma.
  \end{align}
  Then $\{s_\alpha(t)\}$ is an ortho-normal basis of $E$ with metric $h_t$. Recall  the embeddings    $\Phi_t$  of $M$ given by
   $\{s_\alpha(t)\}$.  Define $A_t\in {\rm GL}(E)$ by $\Phi_t=A_t\cdot \Phi_0$. Assume that for a sequence $t_i\rightarrow \infty$, $\Phi_{t_i}(M)$ converges to $\tilde M_\infty$. By Lemma 4.8 in \cite{WZ20}, $(\Phi_{t_i}^{-1})^*\omega_{t_i}$ converges smoothly  to a KR soliton metric $\hat \omega_\infty$ in ${\rm Reg}\,(\tilde M_\infty)$, where $\tilde M_\infty$ is a  Q-Fano variety  with klt singularities. Moreover, we have

\begin{lem}\label{limit-v}
There exists a subsequence $t_i'$ such that  $(\Phi_{t_i'}^{-1})^*\omega_{t_i'}$   converges locally  to a KR soliton  $\hat \omega_\infty$ and
   $A_{t_i'+s}\cdot A_{t_i'}^{-1}$ converges to $e^{sv}$ uniformly,
   where  $v$ is the soliton VF associated to $\hat \omega_\infty$.
  \end{lem}

  \begin{proof}
   It is  known  that $A_{t_i'+s}\cdot A_{t_i'}^{-1}$ is uniformly $C^1$-continuous at $s$ by the partial $C^0$-estimate in \cite{WZ20}.  Then we can choose a subsequence $t_i'$ such that $\lim_{i\rightarrow\infty}A_{t_i'+s}\cdot A_{t_i'}^{-1}=B_s$ uniformly for $s\in [-1,1]$. Thus we have
  \begin{align}\label{fm}
  \lim_{i\rightarrow\infty}\Phi_{t_i'+s} (M )=\lim_{i\rightarrow\infty}A_{t_i'+s}\cdot A_{t_i'}^{-1}(\Phi_{t_i'}M)=B_s\cdot\tilde  M_\infty,
  \end{align}
 where $\tilde  M_\infty$ is a Q-Fano variety limit of $\Phi_{t_i'}(M)$ as in Proposition \ref{two-topology}.

  By Proposition 3.2 in \cite{WZ20}, $(\Phi_{t_i'}^{-1})^*\omega_{t_i'+s}(-1\leq s\leq 2)$  converges  locally uniformly  to $\hat\omega_\infty^s=\tilde \omega_\infty+\sqrt{-1}\partial\bar\partial \psi_\infty^s$ such that  $\psi_\infty^s$ is the solution of
  \begin{align}\label{fl}
  \dot \psi^s_\infty=\log \frac{(\tilde\omega_{\infty}+\sqrt{-1}\partial\bar\partial \psi^s_\infty)^n}{\tilde\omega_{\infty}^n}-\tilde h_\infty+\psi^s_\infty, ~{\rm in}~  {\rm Reg}\,(\tilde M_\infty),
  \end{align}
  where $\tilde \omega_\infty=\frac{1}{l'}\omega_{FS}|_{\tilde  M_\infty}$.
   (\ref{fl}) implies that $\hat \omega_\infty^s$ is a solution of KR flow,
  \begin{align}\label{fl2}
  \frac{\partial}{\partial s}\hat\omega_\infty^s=-{\rm Ric}(\hat \omega_\infty^s)+\hat\omega_\infty^s,~s\in [-1,1],
  \end{align}
  where $\hat \omega_\infty^0=\hat \omega_\infty$.  By Lemma 4.8 in \cite{WZ20}, $\hat\omega^s_\infty$ are all   K\"ahler-Ricci solitons  in ${\rm Reg}\,(\tilde M_\infty)$. Denote the associated soliton VF of  $\hat \omega^s_\infty$ by $\tilde X_s(\tilde X_0=v)$.  Then
  $$\frac{\partial}{\partial s}\hat\omega_\infty^s=-L_{\tilde X_s}\hat\omega_\infty^s.$$
    Considering the one parameter subgroup $F_s$ generated by time dependent vector field $-X_s$, we have $\hat\omega_\infty^s=F_s^*\hat\omega_\infty$. Thus by (\ref{fl2}), we get
  $$L_{\tilde X_s} F_s^*\hat\omega_\infty=F_s^* L_{v}\hat\omega_\infty.$$
    It follows that $(F_s)_* \tilde X_s-v$ is Killing. Hence,  $(F_s)_* \tilde X_s-\tilde X=0$ since it's holomorphic. Taking derivative we obtain  $\tilde X_s=v$, and so
     \begin{align}\label{family-KR-solitons}\hat\omega^s_\infty=F_s^*\hat\omega_\infty=(e^{-sv})^*\hat\omega_\infty.
     \end{align}

     By (\ref{family-KR-solitons}) and  (\ref{fm}),  it is easy to see that
     $$(e^{-s\tilde X})^*\hat\omega_\infty=(B_s^{-1})^*\hat\omega_\infty.$$
     It follows that
  $B_s= g_se^{sX}$ for unitary group $g_s$. Because $B'_sB_s^{-1}$ is unitary, $g_s$ is dependent of $s$. Hence $g_s=g_0=Id$ and $B_s= e^{s\tilde X}$. The lemma is proved.
  \end{proof}

  Now we apply Lemma \ref{limit-v} to any sequence of $\omega_{t_i}$  in a fixed KR flow   (\ref{kr-flow}).  Note that  the soliton VF is $\tilde X=gvg^{-1}$ for some $g\in {\rm U}(N'+1, \mathbb C)$   by Corollary \ref{vec}.  Then   for any sequence $t_i\rightarrow \infty$, there is a subsequence $t_i'\rightarrow \infty$ such that $A_{t_i'+s}\cdot A_{t_i'}^{-1}$ converges to $e^{s v}$ uniformly. Thus we can choose $g_i\in {\rm U}(N'+1, \mathbb C)$ such that $\bar A_i=g_i^{-1}A_ i$ satisfies (\ref{p}) with  $\Lambda=v$, for any $i\in \mathbf N$. Hence, by Proposition \ref{alg}, we get
   \begin{align}\label{reductive-group} [ \mathcal  C_0]\subseteq  \overline {G_v\cdot  [\overline{M}_\infty]},
  \end{align}
  where  $G_{v}=\{g\in {\rm SL}(N'+1,\mathbb C|~g\cdot e^{sv}=e^{sv}\cdot  g\}$, and   $\overline{M}_\infty$ is the limit of $e^{tv} C\cdot \Phi_0(M)(t\rightarrow \infty)$ in $\mathbf {Hilb}^{\mathbb P(E),p}$.

  Denote the Hilbert polynomial of $\Phi_0(M)$ by $p$ and the Hilbert scheme of subschemes of $\mathbb P(E)$ with Hilbert polynomial $p$ by $\mathbf {Hilb}^{\mathbb P(E),p}$. We also have

  \begin{lem}\label{eigen}
   $\overline{M}_\infty$ is normal.  Moreover the eigenvalues and eigenspaces of $\Lambda$ of $H^0(\overline{M}_\infty,\mathcal O(j))$ and $H^0({\tilde M_\infty},\mathcal O(j))$ are the same  for any  $j\geq 1$, where $O(1)=K_{\tilde M_\infty}^{-l}$.
  \end{lem}

  \begin{proof}
      By  (\ref{reductive-group}),  there exists $g_k\in G_v$ such that $\lim_{k\rightarrow \infty}g_k\cdot \overline{M}_\infty=\tilde M_\infty$. Since  the normality is an open condition and $\tilde M_\infty$ is normal, $\overline{M}_\infty $ is normal. Given $j\geq 1$, the  eigenvalues and eigenspaces of $H^0(g_k\cdot \overline{M}_\infty,\mathcal O(j))$ is the same as  that of  $H^0(\overline{M}_\infty,\mathcal O(j))$ for any $k$. Taking limit, the lemma is proved.
  \end{proof}

 \subsection{ Computation of $L(\omega_0')$ }In this subsection, we apply Lemma \ref{eigen} to compute the energy-level $L(\omega_0')$ of KR flow. Recall that
  $$ L(w_0')=\lim_{t\to\infty}\lambda(\omega_t').$$
 Then we have

  \begin{prop}\label{energy-level} Let  $(\tilde M_\infty,  \hat\omega_\infty)$   be a KR soliton limit  of   (\ref{kr-flow})  as in Theorem \ref{WZ}. Then
  \begin{align}\label{max-entropy}\lambda( \hat\omega_{\infty})=\sup\{\lambda(g')|~\omega_{g'}\in 2\pi c_1(M,J)\}.
  \end{align}
 As a consequence, for any  $\omega_0'\in 2\pi c_1(M)$ it holds
 \begin{align}\label{invariant-L}L(\omega')=\sup\{\lambda(g')|~\omega_{g'}\in 2\pi c_1(M,J)\}.
 \end{align}
  \end{prop}

  \begin{proof}The proof is due to \cite{DS16} by  using an argument in \cite{He}.
  Denote a  special  degeneration from $M$ to $\overline{M}_\infty$  by $\mathcal X$ which induced  by  $1$-ps  $e^{sv}$  in Lemma \ref{eigen}. We define an invariant by (cf. \cite{TZ2},  \cite{He}),
     $$H(\mathcal X)=\int_{\overline{ M}_\infty}\bar\theta_v e^{\bar h_\infty} \bar\omega_\infty^n.$$
 As same as the generalized Fuaki-invariant,    one  can show that $H(\mathcal X)$  is independent of choice of admissiable metrics $\bar \omega_\infty$  on $\overline{ M}_\infty$ (cf. \cite{Ti97, BBEGZ}).  Here the Ricci potential  $\bar h_\infty$ of $\bar\omega_\infty$ and potenial $\bar \theta_v$ are normalized by
  $$\int_{\bar M_\infty}e^{\bar h_\infty} \bar\omega_\infty^n=\int_{\bar M_\infty}e^{\bar\theta_v} \bar\omega_\infty^n=V.$$

  By a  result of He \cite{He}  (also see \cite{DS16}),
  $$\sup\{\lambda(g')|~\omega_{g'}\in 2\pi c_1(M,J)\} \le (2\pi)^{-n}[nV-H(\mathcal X)].$$
 Moreover,  as shown in  \cite{DS16}, $H(\mathcal X)$ can be computed using the weights of $v$ on $H^0(\overline{M}_\infty, \mathcal O(j))(j\geq 1)$. Thus by Lemma \ref{eigen}, we have
  $$\int_{\overline{ M}_\infty}\bar\theta_v e^{\bar h_\infty} \bar\omega_\infty^n=\int_{\tilde{ M}_\infty}\tilde\theta_v e^{\tilde h_\infty} \tilde\omega_\infty^n.$$
 Since $\tilde M_\infty$ admits a singular KR soliton $\hat\omega_\infty$ with respect to $v$,
  $$\int_{\tilde{ M}_\infty}\tilde\theta_v e^{\tilde h_\infty} \tilde\omega_\infty^n=\int_{\tilde M_\infty}  \theta_v  e^{ \theta_v } \hat\omega_\infty^n=N_{v},$$
  where
  $$N_{X}= \int_{\tilde M_\infty} \tilde\theta_{X} e^{\tilde\theta_X}  \tilde\omega_\infty^n$$
  is a holomorphic invariant defined  for any holomorphic vector field $X\in {\rm Aut}(M_\infty)$ \cite{TZ2}.
  Thus we get
  $$\sup\{\lambda(g')|~\omega_{g'}\in 2\pi c_1(M,J)\} \le (2\pi)^{-n}[nV-N_{v}].$$
  On the other hand, it was proved in \cite[Lemma 4.14] {WZ20} that
  $$L(\omega_0)=\lambda(\hat\omega_\infty)= (2\pi)^{-n}[nV-N_{v}].$$
  Note that
  $$L(\omega_0)\le \sup\{\lambda(g')|~\omega_{g'}\in 2\pi c_1(M,J)\}.$$
  Hence  (\ref{max-entropy}) must be true. Since   (\ref{max-entropy}) is independent of $\omega_0\in
 2\pi c_1(M)$, we also get (\ref{invariant-L}).
  \end{proof}

 It has been proved by He that "=" in   (\ref{max-entropy}) attains at a KR soliton if $M$ admits such a metric \cite{He}. Proposition  \ref{energy-level} shows that  "=" can   attain  at a singular  KR soliton as a   limit of  KR flow  (\ref{kr-flow}) even without  any  KR soliton  on the original complex manifold $M$.

Now we can generalize the convergence result  \cite[Lemma 4.8]{WZ20} to a sequence of metrics  from varied KR flows as follows.

  \begin{lem}\label{limit-solitons}  Let $\omega^{s_i}\in 2\pi c_1(M,J)$ be a sequence of initial metrics in   (\ref{kr-flow}) converging to $\omega_0'$ in sense of K\"ahler potentials as $s_i\to s_0$. Then for any $t_i\rightarrow \infty$, there is a subsequence $k_i\rightarrow \infty$ such that $\Phi^{s_{k_i}}_{t_{k_i}}(M)$ converges to a normal variety $\tilde M_\infty'$, and $((\Phi^{s_{k_i}}_{t_{k_i}})^{-1})^*(\omega^{s_{k_i}}_{t_{k_i}})$ locally converges to a singular KR soliton $\hat \omega_\infty'$, where  $\Phi^{s_{k_i}}_{t_{k_i}}$ is the Kodaira embedding associated to $\omega^{s_{k_i}}_{t_{k_i}}$.
\end{lem}

\begin{proof}
Since $\omega^{s_i}$ satisfies (\ref{a-condition}), there is a subsequence $k_i\rightarrow \infty$ such that $\Phi^{s_{k_i}}_{t_{k_i}}(M)$ converges to a normal variety $\tilde M_\infty'$. Moreover by Proposition 3.2 in \cite{WZ20}, there is a further subsequence, which is still denoted by $k_i$, such that $({\Phi^{-1}_{t_{k_i}}})^*(\omega^{s_{k_i}}_{t_{k_i}})$  locally converges to a smooth K\"ahler metric $\hat \omega_\infty'$ on ${\rm Reg}(M_\infty')$. By the monotonicity  of Perelman's entropy \cite{Pe, TZ4, WZ20},
  we have
 \begin{align}
&\lim_{t\rightarrow\infty, i\to \infty}\lambda(\omega^{s_{k_i}}_t)\ge L(\omega')=\lim_{i\rightarrow\infty}\lambda(\omega^{s_{k_i}}_{t_{k_i}-1}).
   \end{align}
   By  Proposition \ref{energy-level}, we  also have
$$\lim_{i\rightarrow\infty}\lambda(\omega^{s_{k_i}}_{t_{k_i-1}})=\sup\{\lambda(g')|~\omega_{g'}\in 2\pi c_1(M,J).$$
Thus
$$\lim_{i\rightarrow \infty}(\lambda(\omega^{s_{k_i}}_{t_{k_i+1}})-\lambda(\omega^{s_{k_i}}_{t_{k_i-1}}))=0.$$
Now  we can use the argument in  the proof of  \cite[Lemma 4.8]{WZ20} to prove that $\hat \omega_\infty'$ is a KR soliton in  ${\rm Reg}(\tilde M_\infty')$.
\end{proof}

 As an application of Proposition \ref{energy-level},  the following  gives an analytic character for  a $K$-semistable Fano manifold in terms of Perelaman's entropy.

  \begin{cor}\label{mainThm-lambda} A Fano manifold $M$ is K-semistable if and only if
  \begin{align}\label{KE-entropy}\sup\{\lambda(g')|~\omega_{g'}\in 2\pi c_1(M,J)\}=(2\pi)^{-n}V.
   \end{align}
   \end{cor}

 \begin{proof}[Proof of Corollary \ref{mainThm-lambda}] By \cite[Proposition 4.14]{WZ20} together with Proposition \ref{energy-level}, we see that  $\hat\omega_\infty$   in
  Theorem \ref{WZ} is a singular KE metric.   By Li's result,  Proposition \ref{low},  it follows  that  Ding-energy or K-energy is bounded below. Thus $M$ is K-semistable \cite{Ti97, Ber}.  The inverse part comes from  Li-Sun's result \cite{LS}.

  \end{proof}

 \subsection{Uniform convergence of flow   (\ref{kr-flow})}

     Let  $G_r$  be  the stabilizer of $ \tilde {M}_\infty$  restricted on $G_{v}$. Then it  is reductive by  Proposition \ref{red} in Appendix.
 Thus by  the relative K-polystability of  $\tilde M_\infty'\in  \mathcal C_0$ and (\ref{reductive-group}),  as in the proof of   Theorem \ref{unique-algebra}, we get
 $$\mathcal C_0\subseteq \{G_{v} \cdot \bar M_\infty\}.$$
  It follows that
 \begin{align}\label{unique-alge-variety}\mathcal C_0\subseteq \{{\rm SL}(N+1,\mathbb C) \cdot \bar M_\infty\},\end{align}
 which was  proved in \cite{CSW}.

 By the uniqueness of   $\tilde M_\infty'$ in (\ref{unique-alge-variety}),   we prove the following  uniform convergence of flow   (\ref{kr-flow}).

   \begin{theo}\label{mainThm-KS}Let $(\tilde M_\infty, \hat \omega_\infty)$ be  a singular KR soliton limit of  a sequence  $(M, \omega_{t_i})$ of  (\ref{kr-flow})  as in Theorem \ref{WZ}.
     Then  $\omega_t$ is locally uniformly  convergent to  $\hat\omega_\infty$ on   ${\rm Reg}(\tilde M_\infty)$ in the Cheeger-Gromov topology. Moreover,  $(M, \omega_t)$ is  uniformly convergent to a  Gromov-Hausdroff limit   $(M_\infty, \omega_\infty)$ in HT-conjecture, which  is
    the completion of $({\rm Reg}(\tilde M_\infty), \hat\omega_\infty)$.
    \end{theo}

   \begin{proof}By  (\ref{unique-alge-variety}),  we see that submanifolds $\tilde M_t\subset \mathbb CP^N$  converge locally uniformly to a   Q-Fano variety $\tilde M_\infty$ with klt singularities which admits a    weak KR soliton $\hat \omega_\infty$. As in Section 2, we can choose  an exhausting open sets $\Omega_\gamma\subset  \tilde M_\infty$.  Then   there are  diffeomorphisms $\Psi_\gamma^t: \Omega_\gamma\to \tilde M_t$ such that
   the curvature of $ \omega_{FS}|_{\tilde {\Omega}_\gamma^t}$ is $C^k$-bounded  uniformly independently of $t$ , where  $\tilde {\Omega}_\gamma^t=\Psi_\gamma^t(\Omega_\gamma)$.  Write as
   \begin{align}\label{psi-s-t}(\Phi_t^{-1})^*\omega_{t+s}=\tilde\omega_{t}+\sqrt{-1}\partial\bar\partial \psi^s_t,   ~{\rm in} ~\tilde M_t, \forall s\in [-1,1],
   \end{align}
   $\tilde\omega_t=\frac{1}{l}\omega_{FS}|_{\tilde M_t}$.
   Thus an in the proof of  Proposition 3.2 in \cite{WZ20},  K\"ahler potential $\psi^s_t$ in Lemma \ref{c3-phi} will satisfy (\ref{c0-estimate})-(\ref{c3-estimate}) on $\tilde {\Omega}_\gamma^t$.  Those estimates implies that $\omega_t$ is locally uniformly  convergent to  $\hat\omega_\infty$ on   ${\rm Reg}(\tilde M_\infty)$ in the Cheeger-Gromov topology.

   The second part in the theorem follows from the fact
   $$(M_\infty, \omega_\infty)=\overline{({\rm Reg}(M_\infty), \hat\omega_\infty)}$$
   as proved in Theorem \ref{WZ}.  By the above local convergence, this implies that any Hausdroff-Gromov limit of sequence of $(M, \omega_t)$ is same, thus the convergence is uniform.

 \end{proof}

 \subsection{Uniqueness of  KR soliton  VFs-extension}

 In this subsection, we use Lemma \ref{limit-solitons} together with Proposition \ref{two-topology} to prove an analogy of Corollary \ref{vec} for the uniqueness of KR soliton VFs  associated to limits of   (\ref{kr-flow})  with varied initial metrics.

    Fix  a sequence  $s_i\to 0$.  Then by Lemma \ref{limit-solitons}, we see that for any sequence of $t_i\to\infty$   there exists a subsequence of  $t_i'$ such that  $(\Phi_{t_i'}^{-1})^*\omega_{t_i'}^{s_i}$  locally converges to a singular KR soliton  $\hat\omega_\infty'$ on a
   Q-Fano variety $\tilde M_\infty'$ with klt singularities,
   where $\Phi_{t_i'}$ is the Kodaira embedding associated to the metric  $\omega_{t_i'}^{s_i}$.  As in Section 1, we set  a class of  all  possible limits of Q-Fano varieties $\tilde M_\infty'$   by $ \mathcal C_A$, which consists of
    Q-Fano varieties with admitting  singular KR solitons  as  local limits  of sequences  of metrics  from KR flows with initial metrics satisfying  (\ref{a-condition}).

  \begin{prop}\label{vector-gap} There is a small $\epsilon$ such that for any $(\tilde M_\infty, \hat \omega_\infty, X),    (\tilde M_\infty', \hat \omega_\infty', X')\in  \mathcal C_A$, it holds
  \begin{align}\label{vf-unique-0}
 \tilde X'= g\cdot \tilde X \cdot g^{-1}
 \end{align}
 for some $g\in {\rm U}(N' + 1;\mathbb C)$, if
   \begin{align}\label{small-GH}{\rm dist}_{GH}( ( M_\infty, \omega_\infty),    ( M_\infty', \omega_\infty'))\le \epsilon,
   \end{align}
  where  $(M_\infty, \omega_\infty) =\overline{(\tilde M_\infty, \hat \omega_\infty)}, (M_\infty', \omega_\infty') =\overline{(\tilde M_\infty', \hat \omega_\infty' )}$ are compactified Gromov-Hausdroff spaces as in Proposition \ref{two-topology}.

   \end{prop}

  \begin{proof} First we note that according to  the proof of  Proposition \ref{continuity-vf}  the set of soliton VFs  is compact  in ${\rm sl}(N'+1, \mathbb C)$. Since there are countably many subtori of $(\mathbb C^*)^N$, by Lemma \ref{vec1},  we need to prove that for any  sequence of  soliton VFs  $X^i$ associated to  singular KR solitons ${(\tilde M_\infty^i, \hat \omega_\infty^i)}$  in  $\mathcal C_A$ which converge  to   $(M_\infty, \omega_\infty)$ in Gromov-Hausdroff topology, it holds
  \begin{align}\label{vf-convergenvce}X^i\to  g\cdot \tilde X \cdot g^{-1}, ~{\rm for ~some} ~g\in {\rm SU}(N' + 1;\mathbb C).
  \end{align}

  As in the proof of Proposition \ref{two-topology} together with Lemma \ref{limit-solitons},     there is a  sequence of    $\omega^{\alpha_i}_{i_k}$   and
  a   Q-Fano variety $\bar M_\infty$ such that  $\omega^{\alpha_i}_{i_k}$ is  $C^\infty$-convergent to a  singular KR soliton   $(\bar \omega_\infty, \bar X)$  on  ${\rm Reg}(\bar M_\infty)$.  Moreover,
     \begin{align}\label{isometric-2}\overline {(\bar\omega_\infty, {\rm Reg}(\bar M_\infty))}\cong ( M_\infty, \omega_\infty).
     \end{align}
  Since
 $$\int_{ {\rm Reg}(\bar M_\infty)} \bar\omega_\infty^n=\int_{ {\rm Reg}(\tilde M_\infty)} \hat\omega_\infty^n=V,$$
 it is easy to see that there are  an open sets  $\bar U\subset   {\rm Reg}(\bar M_\infty)$ and   $U\subset  {\rm Reg}(\tilde M_\infty)$ such that
 $$(\bar\omega_\infty, \bar U)\cong  (\hat \omega_\infty,  U).$$
 The above also implies that  both of complex structures on  $\bar U$ and $U$ are same by the convergence of metric sequences.
 Thus one can extend  the metric $\bar\omega_\infty|_{\bar U}$ to a KR soliton $( \bar\omega_\infty', \bar X')$  on  ${\rm Reg}(\tilde M_\infty)$,
  where $\bar X'\in \eta_\infty$ as an element of Lie algebra of ${\rm Aut}(\tilde M_\infty)$.   By the uniqueness,
  $\bar\omega_\infty'$ is same as $\hat \omega_\infty$ in sense of solutions of weak complex Monge-Amp\`ere equation associated to singular KR solitons (cf. \cite{WZ20}). It follows that
 $$\bar X=\bar X'= \sigma_*(\tilde X) $$
 for some $\sigma\in {\rm Aut}(\tilde M_\infty)$.
 Thus there is some $g\in {\rm SL}(N' + 1;\mathbb C)$ such that
 $$\bar X= g\cdot \tilde X \cdot g^{-1}.$$
 This proves  (\ref{vf-convergenvce}).

  \end{proof}

 \section{Proofs of Theorem \ref{general-KR}  and Theorem \ref{stability-KR-smooth-complex}}

 To prove  Theorem \ref{general-KR}, we need  a stability result  for  KR flow (\ref{kr-flow}) with  a limit of singular  KR soliton    $(\tilde M_\infty, \hat \omega_\infty)$  analogous to
   Theorem \ref{stability-KE}. In fact,
 by assuming  that  ${\rm Aut}_0(\tilde M_\infty)$ is reductive,    we have

 \begin{lem}\label{stability-KR}   Let   $(\tilde M_\infty, \hat\omega_\infty)$ be the   singular  KR soliton  as the  limit of flow   (\ref{kr-flow}).  Suppose that  ${\rm Aut}_0(\tilde M_\infty)$ is reductive.  Then   there is an $\epsilon>0$ such that for any initial metric  $\omega_0'\in 2\pi c_1(M,J)$ with (\ref{small-initial-0}) satisfied,
 flow $(M, \omega_t')$  is uniformly   convergent to a  Gromov-Hausdroff limit   $(M_\infty, \omega_\infty)$,  which  is
    the completion of $({\rm Reg}(\tilde M_\infty), \hat\omega_\infty)$.
 \end{lem}

 \begin{proof}
 Under the condition that $ {\rm Aut}(\tilde M_\infty)$ is reductive,   we    can modify the proof of    of Theorem \ref{stability-KE} to  prove  Lemma  \ref{stability-KR}.
    In fact,  by Proposition \ref{two-topology},
      Proposition \ref{dis2}  can be applied to  a pair of $Q$-Fano varieties with   one reductive.
    We note that
   as in the proof of   Lemma \ref{limit-solitons} together with Proposition  \ref {energy-level} the sequence $\{\omega_{t_i}^{\alpha_i}\}$  ($t_i\to\infty, \alpha_i\to 0)$ obtained  by the contradiction argument as in (\ref{contradiction-2}) will converge to another  singular KR soliton $(\tilde M_\infty', \omega_\infty')$ with
   \begin{align}\label{gh-small-distance-2} \frac{\epsilon}{2}\le {\rm dist}_{GH}(  (\tilde M_\infty, \omega_{KR}),    (M_\infty', \omega_\infty'))\le \epsilon<<1.
   \end{align}
   Moreover,   by  Lemma \ref{vector-gap},  the soliton  VFs  $X^i$ associated to  singular KR solitons ${(\tilde M_\infty^i, \hat \omega_\infty^i)}$ is conjugate to $v$.   Thus   by  Proposition \ref{dis2}  $M_\infty'$ must conjugate with $\tilde M_\infty$, and so  $(M_\infty', \omega_\infty')$ must be isometric to  $(M_\infty, \omega_\infty)$ by the uniqueness of KR solitons.   But the latter  is  contradict  with  (\ref{gh-small-distance-2}).   The proof is    finished.

   \end{proof}

  \begin{proof}[Proof of Theorem \ref{general-KR}]We use the idea in Section 4 to prove  Theorem \ref{general-KR}.  As in  the proof of Theorem \ref{mainThm-KE},  for any $\omega_0'\in 2\pi c_1(M,J)$,   we let $\omega^s=s\omega_0+(1-s)\omega_0'$ ($s\in [0,1]$).   We want to prove  the uniform convergence of  flow  $(M, \omega_t^s)$ for any initial $\omega^s$  in   Cheeger-Gromov topology. Let  $I$ be a set as in (\ref{i-set}).  Then   $I$ is non-empty and open  by Lemma \ref{stability-KR}.   We remains to prove that $I$ is also closeness. Without of
  loss of generality, we may assume that $I=[0, s_0)$  and we are going to show that $s_0\in I$.

  Let $L(\omega^{s_0})$ be  the energy level  with respect to $\omega^{s_0}$ as in (\ref{level}). Then by Proposition  \ref {energy-level}, we have
    \begin{align}\label{l-condition}
    L(\omega^{s_0})=   L(\omega_{0})=\lambda(\omega_{\infty}).
    \end{align}
  We claim  that for any $\delta>0$ there are  an $\epsilon_0>0$ and $t_0$ such that
   \begin{align}\label{small-GH-2}{\rm dist}_{GH}( ( M_\infty, \omega_{\infty}),    (M, \omega_t^s))\le \delta,~\forall~s\in [s_0-\epsilon_0, s_0), t\ge t_0.
   \end{align}
   As a consequence,  we get
   $${\rm dist}_{GH}( ( M_\infty, \omega_{\infty}),   (M_\infty', \omega_\infty'))\le \delta,
   $$
   where $(M_\infty', \omega_\infty')$   is the     Gromov-Hausdroff  limit  of $(M, \omega_t^{s_0})$  by Theorem  \ref{unique-algebra}.
 Thus the theorem will follow from Proposition \ref{dis2} together with Proposition \ref{two-topology} as in the proof of  Theorem \ref{mainThm-KS} since $ {\rm Aut}(\tilde M_\infty)$ is reductive.

  We   prove (\ref{small-GH-2})  by contradiction  as in the proof of  Theorem \ref{stability-KE}.  On the contrary,  for  a small number $\delta_0$  $(=\epsilon)$ as chosen  in (\ref{small-GH}) in  Proposition \ref{two-topology},  there is  a sequence of  $\omega_{t_i}^{s_i}$ ~( $t_i\to\infty$,  $s\to s_0$)  such that
 \begin{align}\label{contradiction} \frac{\delta_0}{2}\le {\rm dist}_{GH}(  (M_\infty, \omega_{\infty}),   \omega_{t_i}^{s_i} )\le \delta_0.
 \end{align}
 Note that
  \begin{align}\label{max-entropy-comutation}&\lim_{t_i\rightarrow\infty, s_i\to s_0}\lambda(\omega_{t_i}^{s_i})\ge L(\omega^{s_0})=\lim_{t\rightarrow\infty}\lambda(\omega_{t}) \notag\\
 & =\sup\{\lambda(g')|~\omega_{g'}\in 2\pi c_1(M,J)\}.
   \end{align}
 It follows that
 $$\lim_{t_i\rightarrow\infty, s_i\to s_0}\lambda(\omega_{t_i}^{s_i})=L(\omega^{s_0}).$$
 Thus    as in the proof of   Lemma \ref{limit-solitons},     $\{\omega_{t_i}^{s_i}\}$ is  locally $C^\infty$-convergent to a  KR soliton on the regular part of  a Q-Fano variety $\tilde M_\infty'$ with klt singularities.
 Moreover, its  Gromov-Hausdroff  limit  $(\bar M_\infty, \bar \omega_\infty))$ satisfies
 \begin{align}\label{contradiction-4}  \frac{\delta_0}{2}\le {\rm dist}_{GH}(  (M_\infty, \omega_{\infty}),    (\bar M_\infty, \bar \omega_\infty))\le \delta_0.
 \end{align}
 Hence,  by  Proposition \ref{dis2}  together with Proposition   \ref{two-topology}, we conclude that $(\bar M_\infty, \bar \omega_\infty)$  is isometric to   $(M_\infty, \omega_{KR})$  as in the  proof of  Theorem \ref{mainThm-KS}.  This is impossible by (\ref{contradiction-4}). Therefore, (\ref{small-GH-2}) is true and  Theorem \ref {stability-KE} is proved.
 \end{proof}

  \subsection {Globally smooth convergence}
  In this  subsection,  we further  assume that the Q-Fano variety  limit $ \tilde M_\infty$ in Theorem \ref{WZ} is smooth.
  Then   there are a covering $\{U_\alpha\}$  of $ \hat\omega_\infty$  with local  holomorphic coordinates and   diffeomorphisms   $\Psi^{i}: \tilde M_\infty\to \tilde M_i$ such that for each $\tilde M_i$ there is a  covering  $\{U_\alpha^i\subset \Psi^{i}(U_\alpha)\}$   with local  holomorphic coordinates and uniform norms of transformation functions. Thus if we
     write as
   \begin{align}\label{psi-s-smooth}(\Phi_i^{-1})^*\omega_{t_i+s}=\tilde\omega_{i}+\sqrt{-1}\partial\bar\partial \psi^s_i,   ~{\rm in} ~\tilde M_i, \forall s\in [-1,1],
   \end{align}
     K\"ahler potential $\psi^s_i$ will satisfy  (\ref{c0-estimate})-(\ref{c3-estimate}) in Lemma \ref{c3-phi} on each $U_\alpha^i$.  Those estimates imply that $\omega_{t_i}$ is  $C^\infty$-convergent to a smooth  KR soliton  $(\hat\omega_\infty, \tilde M_\infty)$ in Cheeger-Gromov topology, and   so  the Gromov-Hausdroff limit $\omega_{KS}$ is asme as    $\hat\omega_\infty$ and $M_\infty$ is diffeomorphic to $\tilde M_\infty$.  Actually, we have

 \begin{lem}\label{smoothcase-lemma} Let   $(M_\infty, \omega_\infty)$   be a  Gromov-Hausdroff limit  of sequence   $\{\omega_{t_i}\}$ in the KR flow (\ref{kr-flow}). Then $(M_\infty, \omega_\infty)$   is a smooth KR soliton  if and only if the  Q-Fano variety $\tilde M_\infty$   is smooth and it is  diffeomorphic to $ M_\infty$.
  %as in Theorem \ref{WZ}

 \end{lem}

 \begin{proof}We need to prove the necessary part. In fact, by (\ref{yau-equation-limit}), we have
 \begin{align}\label{yau-equation-limit-2}
  (\omega_\infty+\sqrt{-1}\partial \bar \partial \kappa)^n=e^{ h_\infty} \omega_\infty^n,~{\rm in}~ \tilde M_\infty,
  \end{align}
  where $h_\infty$ is  a Ricci potential of $\omega_\infty$.
  We claim $\kappa$ can be extended to a smooth solution of (\ref{yau-equation-limit-2}) on  $M_\infty$. This implies that  the modified K\"ahler metrics $\eta_t$ of $\omega_t$  in (\ref{modified-metric}) converges to a smooth limit of Gromov-Hausdorff \cite{WZ20}. In particular, each tangent cone at $p\in (M_\infty,\omega_\infty+\sqrt{-1}\partial \bar \partial \kappa)$ is flat.   Thus  by Proposition 2.4 in \cite{LS} (also see (4.31) in \cite{WZ20}), $\tilde M_\infty$ is smooth. Hence, the convergence of  $\omega_{t_i}$ is  $C^\infty$ in Cheeger-Gromov topology and so $ M_\infty$ is  diffeomorphic to $ \tilde M_\infty$. The lemma is proved.

  Since $\kappa$ is uniformly bounded on $\tilde M_\infty$, it is a globally weak solution of (\ref{yau-equation-limit-2}) on  $M_\infty$.
  On the other hand, by Yau' s theorem to Calabi's conjecture, there is a  smooth  solution $\kappa'$ of  (\ref{yau-equation-limit-2}) on  $M_\infty$. Thus by the uniqueness of weak solutions,
  $$\kappa=\kappa'+c,$$
  for some constant $c$. Hence, $\kappa$ must be a smooth  solution  on  $M_\infty$.

 \end{proof}

 \begin{cor}\label{smooth-KR}
 Suppose that  there is a sequence   $(M, \omega_{t_i})$ of  (\ref{kr-flow}) whose   Gromov-Hausdroff limit   is  a smooth   KR soliton $(M_\infty,  \omega_{KR})$  with  ${\rm Aut}_0(\tilde M_\infty)$ reductive. Then the flow is uniformly $C^\infty$-convergent to $(M_\infty,  \omega_{KR})$  in Cheeger-Gromov topology and the convergence is independent of  initial metric $\omega_0'\in 2\pi c_1(M)$.
 \end{cor}

 \begin{proof}  By Lemma \ref{smoothcase-lemma},  $(M_\infty, \omega_\infty)$   is a smooth KR soliton  and the  Q-Fano variety $\tilde M_\infty$   is smooth.  Then, by Theorem \ref{general-KR}, for  Ricci flow $\omega_t'$ with  any initial metric $\omega_0'\in 2\pi c_1(M)$,  there are  diffeomorphisms $\Psi^t: \tilde M_\infty\to \tilde M_t'=\Phi_t(M)$ such that for each $\tilde M_t'$ there is a  covering  $\{U_\alpha^t\subset \Psi^{t}(U_\alpha)\}$   with local  holomorphic coordinates and uniform norms of transformation functions.
    Write as
  $$(\Phi_t^{-1})^*\omega_{t+s}'=\tilde\omega_{t}'+\sqrt{-1}\partial\bar\partial {\psi^s_t}',   ~{\rm in} ~\tilde M_t', \forall s\in [-1,1],
   $$
   where $\tilde\omega_t'=\frac{1}{l}\omega_{FS}|_{\tilde M_t'}$.  Thus the estimates for  ${\psi^s_t}'$ in Lemma \ref{c3-phi} on each $U_\alpha^t$  imply that $\omega_{t}'$ is  $C^\infty$-convergent to $(\hat\omega_\infty, \tilde M_\infty)$ in Cheeger-Gromov topology.
 \end{proof}

 \subsection{Proof of Theorem \ref{stability-KR-smooth-complex}}

  According to the proof of Corollary \ref{smooth-KR}, the induced  metric $(\Phi_t^{-1}\cdot \Psi^t))^*\omega_{t}$ of (\ref{kr-flow}) is  $C^\infty$-convergent to a KR soliton $(M_\infty,  \omega_{KR})$ with the complex structure $J_\infty$ defined by
  $$ J_\infty =\lim_{t\to \infty}(\Phi_t^{-1}\cdot \Psi^t)^*J.$$
  Thus  $J_\infty$  is a  canonical    smooth deformation  of $J$.   Conversely, by  the first relation in (\ref{c-infty-convergence-vector}) implies that  the curvature of $\omega_i$ is uniformly bounded. Then by the partial $C^0$-estimate for the  sequence  of $\omega_i$ (cf. \cite{DS, Ti13, T2}), there are images $\tilde M_i\subset  \mathbb CP^N$  of Kodaira embeddings as in Section 1,  which converges to a smooth submanifold  $\tilde M_\infty $ such that (\ref{group-g}) holds. Thus
  $$\tilde M_\infty\in \overline{{\rm SL}(N + 1;\mathbb C)\cdot\tilde M}. $$

 \begin{proof}[Proof of Theorem \ref{stability-KR-smooth-complex}] Let $\omega_i$ be a sequence of K\"ahler metrics  in $2\pi c_1(M,J)$ as in (\ref{deformation-j})  such that
 $$\lim_i{\rm dist}_{CH}((M,\omega_i), (M', \omega'))=0.$$
 By  Corollary \ref{smooth-KR}, it suffices to prove: for any $\delta>0$
  there are $i_0$ and $t_0$ such that
   \begin{align}\label{small-GH-4}{\rm dist}_{GH}( ( M', \omega_{KS}),    (M, \omega_t^i))\le \delta,~\forall~i\ge i_0, t\ge t_0,
   \end{align}
   where $\omega_t^i$ is the solution of  (\ref{kr-flow}) with the initial metric  $\omega_i$.
   As a consequence,  we get
   $${\rm dist}_{GH}( (M', \omega_{KR}),   (M_\infty', \omega_\infty'))\le \delta,
   $$
   where $(M_\infty', \omega_\infty')$   is the  global   Gromov-Hausdroff  limit  of $(M, \omega_t^i)$  by Theorem \ref{WZ}.
 Thus by Proposition \ref{dis2} together with Proposition \ref{two-topology} we get
  $$(M', \omega_{KR}) \cong (M_\infty', \omega_\infty'),$$
  since $ {\rm Aut}( M')$ is reductive.  The theorem is proved.

  We   prove (\ref{small-GH-4})  by contradiction  as in the proof of  Theorem \ref{stability-KE}.  First we note that
   the KR flow (\ref{kr-flow}) with the initial metric $\omega'$ on $(M',J')$ is  uniformly  $C^\infty$-convergent to    $(M', J',\omega_{KR})$ by Corollary \ref{smooth-KR}. Then, on the contrary,  for  a small number $\delta_0$  $(=\epsilon)$ as chosen  in (\ref{small-GH}) in  Proposition \ref{two-topology},  we can find  a sequence of  $\omega_{t_i}^{i}$ ~( $t_i\to\infty$,  $i\to\infty$)  such that
 \begin{align}\label{contradiction} \frac{\delta_0}{2}\le {\rm dist}_{GH}(  (M', \omega_{KR}),   \omega_{t_i}^{i} )\le \delta_0.
 \end{align}

 On the other hand, by the monotonicity of Perelman's entropy together with the condition (\ref{max-entropy-4}),  it is easy to see that
 $$\lim_{i\to\infty}\lambda(\omega_{t_i}^{i})=L(\omega').$$
 Thus      as in the proof of   Lemma \ref{limit-solitons},     $\{\omega_{t_i}^{i}\}$ is  locally $C^\infty$-convergent to a  K\"ahler-Ricci soliton on the regular part of  a Q-Fano variety with klt singularities.
 Moreover, its  Gromov-Hausdroff  limit  $(\bar M_\infty, \bar \omega_\infty)$ satisfies
 \begin{align}\label{contradiction-3}  \frac{\delta_0}{2}\le {\rm dist}_{GH}(  (M', \omega_{KR}),    (\bar M_\infty, \bar \omega_\infty))\le \delta_0.
 \end{align}
 Hence,  by  Proposition \ref{dis2}  together with Proposition   \ref{two-topology}, we conclude that $(\bar M_\infty, \bar \omega_\infty)$  is isometric to   $(M', \omega_{KR})$  as in the  proof of  Theorem \ref{mainThm-KS}.  This is impossible by (\ref{contradiction-3}). Therefore, (\ref{small-GH-2}) is  proved.
 \end{proof}

 \begin{rem}\label{generalization-smooth}
 Since the  singular KR soliton  $\overline{(\tilde M_\infty', \hat \omega_\infty')}$ of KR flow on $(M', J')$  is unique by
  Theorem \ref{general-KR}, we can generalize Theorem \ref{stability-KR-smooth-complex} as follows: Let $(M', J')$  be   a canonical
  smooth deformation  of a Fano manifold  $(M,J)$. Suppose that  ${\rm Aut}_0(\tilde M_\infty')$ is reductive and
  $$\lambda(\tilde \omega_{\infty})=\sup\{\lambda(g')|~\omega_{g'}\in 2\pi c_1(M,J)\}.$$
  Then   for any initial metric  $\omega_0'\in 2\pi c_1(M,J)$
  KR flow $(M, J, \omega_t')$  uniformly  converges to     $\overline{(\tilde M_\infty', \hat \omega_\infty')}$ in Gromov-Hausdroff topology.

 \end{rem}

 \begin{rem}\label{counter-example}
 In \cite{Pas},  Pasquier showed that  the Grassman manifold
 $Gr_q(2,7)$  can be deformed  to a  horo-spherical manifold $(M', J')$.  By the stability of K\"ahler metrics \cite{Kod},  $(M', J')$ is a  jump of  $Gr_q(2,7)$.  On the other hand,
  by a recent result of Deltroix,   any horo-spherical Fano manifold admits a KR soliton \cite{Del}.  Since   $(M_\infty, J_\infty)$  has non-vanishing Futaki-invariant,    $(M_\infty, J_\infty)$ admits a (non-KE)  KR soliton.
  Clearly $(Gr_q(2,7), J)$ admits a KE
  metric as a symmetric space and any KR flow on  $Gr_q(2,7)$  converges uniformly to the KE metric in Cheeger-Gromov topology \cite{TZ3} (in fact in sense of  K\"ahler potentials modulo ${\rm Aut}(Gr_q(2,7))$. Thus KR flow  could not be stable near the KR soliton on  $(M', J')$ when
 the complex structure  varies  from $J$ to $J'$.  In particular, Theorem \ref{stability-KR-smooth-complex} is not true  for $(M', J')$. The reason is that  (\ref{max-entropy-4}) does not hold,
 $$\lambda(\omega_{KR})<\sup\{\lambda(g')|~\omega_{g'}\in 2\pi c_1(M,J)\}=\lambda(\omega_{KE}).$$

 \end{rem}

  Corollary \ref{uniqueness-orbit} is a direct consequence of Theorem \ref{stability-KR-smooth-complex},
 since KR flow $(M,\omega_t^{i,1})$  with initial metrics $\omega_i^1 $  (or $\omega_i^2)$ converges uniformly to $(\tilde M_\infty^1, \omega_{KR}^1)$  (or $(\tilde M_\infty^2, \omega_{KR}^2)$) as $i>>1$.  By Theorem \ref{general-KR},  $(\tilde M_\infty^1, \omega_{KR}^1)$ and $(\tilde M_\infty^2, \omega_{KR}^2)$ must be same.

 \subsection{Further remarks on Theorem \ref{stability-KR-smooth-complex} and  Corollary \ref{uniqueness-orbit}}
 To generalize   $(M', J')$ to a  singular limit in  Defintion \ref{deformation-j},  we  introduce

 \begin{defi}\label{deformation-j-2}Let $(M,J)$ be a Fano manifold. A $Q$-Fano variety   $(M',J')$  is called a canonical  deformation of  $(M,J)$ with bounded Ricci curvature
 if there is  a sequence of K\"ahler metrics $\omega_i$ in $2\pi c_1(M,J)$
 such that
 \begin{align}\label{c-infty-convergence-vector-2} &|{\rm Ric}(\omega_i)|\le \Lambda,~{\rm Vol}(B_1(p_i), \omega_i)\ge c_0,\notag\\
  &(M, \omega_i)     \stackrel{GH}{\longrightarrow} (M', d').
 \end{align}
 \end{defi}

 By Cheeger-Colding-Tian' s theorem \cite{CCT},  $(M', d')$ can be decomposed into the regular part $M_1'$,  a $C^{1, \alpha}$-Riemanian manifold  and the singular part $M_2'$ with  the Hausdroff measure of at least  codimension 4.
  Moreover, by  the partial $C^0$-estimate \cite{ DS, T2, JWZ}, the  $Q$-Fano structure    $(M',J')$ is  given as a limit of smooth submanifolds $\tilde M_i$ as in Section 1. In addition that  $(M',J')$ admits a singular KR soliton $(M', \omega_{KR})$, $(M',J')$ has klt singularities \cite{BBEGZ, WZ20}. Thus by the technique of MA equation,   as in Lemma \ref{c3-phi},  we get
   \begin{align}
   &| \psi_i|\le A, ~{\rm in} ~\tilde M_i, \label{c0-estimate-11}\\
   &C_\gamma^{-1}\tilde\omega_{i}\le(\Phi_i^{-1})^*\omega_{i}\le C_\gamma \tilde\omega_{i}, ~{\rm in}~\tilde \Omega_\gamma^i,
   \label{c2-estimat-12e} \\
   &\| \psi_i\|_{C^{3,\alpha}(\tilde {\Omega}_\gamma^i)} \le A, \label{c3-estimate-2}
   \end{align}
 where $\psi_i$ is the K\"ahler potential of $\omega_i$ associated to the background
 $$\tilde\omega_i=\frac{1}{l}\omega_{FS}|_{\tilde M_i}.$$
 As a consequence, we get an open  $C^{1,\alpha}$ K\"ahler metric $\hat \omega_\infty$ on  ${\rm Reg}(M')$ which satisfies
  \begin{align}\label{compactifcation} \overline{({\rm Reg}(M'),  \hat\omega_{\infty})}= (M', d').
  \end{align}

 It is interesting in constructing  an approximation of  singular KR soliton  $\omega_{KR}$  by $\hat\omega_i$ which satisfies (\ref{c-infty-convergence-vector-2}) with  the compactification of  $\omega_{KR}$ as its Gromov-Hausdroff limit.  Then the  K\"ahler potential of  $\hat\omega_i$  will satisfy
 (\ref{c0-estimate-11}), (\ref{c2-estimat-12e}) and (\ref{c3-estimate-2}).  As a consequence, we can generalize  Theorem \ref{stability-KR-smooth-complex} to  the case of $Q$-Fano variety $(M',J')$ with singular KR solitons  which is  a canonical  deformation of  $(M,J)$ with bounded Ricci curvature.  Actually,  the above argument implies the following uniqueness result as a generalization of    Corollary \ref{uniqueness-orbit} (also see Remark \ref{generalization-smooth}).

 \begin{theo}\label{uniqueness-orbit-singular}  Let $\{\omega_i^1\}$  and $\{\omega_i^2\}$ be two  sequences of K\"ahler metrics  in $2\pi c_1(M, J)$ which satisfy (\ref{c-infty-convergence-vector-2}) with  Gromov-Hausdroff limits  compactified by two  singular KR solitons  $(M_\infty^1, J^1, \omega_{KR}^1)$ and  $(M_\infty^2, J^2, \omega_{KR}^2)$ as in (\ref{compactifcation}), respectively.  Suppose that ${\rm Aut}_0( M_\infty^1)$ and  ${\rm Aut}_0( M_\infty^2)$ are both  reductive, and
 \begin{align}\label{max-entropy-0}\lambda(\omega_{KR}^1)=\lambda(\omega_{KR}^2)=\sup\{\lambda(g')|~\omega_{g'}\in 2\pi c_1(M,J)\}.
 \end{align}
  Then  $(M_\infty^1, J^1)$ is biholomorphic to  $(M_\infty^2, J^2)$. Moreover,
  $$({\rm Reg}(M_\infty^1),\omega_{KR}^1)\cong ({\rm Reg}(M_\infty^2),\omega_{KR}^2), $$
 and consequently,
   $$\overline{({\rm Reg}(M_\infty^1),\omega_{KR}^1)}\cong\overline{({\rm Reg}(M_\infty^2),  \omega_{KR}^2)}.$$

  \end{theo}

  \section{Appendix}
  \subsection{Reductivity of ${\rm Aut}_0(\tilde M_\infty)$}

  \begin{prop}\label{red}Suppose that the limit $(\tilde M_\infty, \hat\omega_\infty)$ in Theorem \ref{WZ} is  a singular  KE metric.
   Then ${\rm Aut}_0(\tilde M_\infty)$ is reductive.
  \end{prop}

  We use the method in \cite{T2} to prove Proposition \ref{red}.   First we  estimate the first non-zero eigenvalue $\lambda_1({t_i})$ of Laplace operator associated to
  $\omega_{t_i}$. We have

  \begin{lem}\label{eigenvalue}
 $$\underline{\lim}_{i\rightarrow \infty}\lambda_1({t_i})\geq 1.$$
  \end{lem}

  \begin{proof}
  We claim:
  \begin{align}\label{ricl}
  \int_M|{\rm Ric}\,(\omega_{t_i})-\omega_{t_i}|^2\omega_{t_i}^n\rightarrow 0.
  \end{align}

  By Lemma \ref{c3-phi}, we have
  \begin{align}
  \int_{\tilde \Omega^i_\gamma}|\nabla h_{t_i}|^2\omega_{t_i}^n\rightarrow 0. \notag
  \end{align}
  Combined with Lemma \ref{lem:perelman-1}, we get
  \begin{align}\label{gr}
  \int_{M}|\nabla h_{t_i}|^2\omega_{t_i}^n\rightarrow 0.
  \end{align}
  Similarly  we have
  $$\int_M |R(\omega_{t_i})-n|^2\omega_{t_i}^n\rightarrow 0.$$
On the other hand,   by the Bochner formula,  we have
  $$\Delta |\nabla h_{t_i}|^2={\rm Ric}\,(\omega_{t_i})(\nabla h_{t_i},\nabla h_{t_i})+|{\rm Hess}\,h_{t_i}|^2+\langle \nabla h_{t_i},\nabla \Delta h_{t_i}\rangle .$$ Integrating both sides, we get
  \begin{align} &\int_M |{\rm Hess}\,h_{t_i}|^2\omega_{t_i}^n+\int_M{\rm Ric}\,(\omega_{t_i})(\nabla h_{t_i},\nabla h_{t_i})\omega_{t_i}^n\notag\\
  &=\int_M (\Delta h_{t_i})^2 \omega_{t_i}^n=\int_M |R(\omega_{t_i})-n|^2\omega_{t_i}^n\rightarrow 0.\notag
  \end{align}
  By Lemma \ref{lem:perelman-1}, the second term can be estimated as
  \begin{align}
  |\int_M{\rm Ric}\,(\omega_{t_i)}(\nabla h_{t_i},\nabla h_{t_i})|\omega_{t_i}^n &\leq \int_M |{\rm Hess}\,h_{t_i}||\nabla h_{t_i}|^2\omega_{t_i}^n+\int_M|\nabla h_{t_i}|^2\omega_{t_i}^n\notag \\
  &\leq C\int_M |{\rm Hess}\,h_{t_i}||\nabla h_{t_i}|\omega_{t_i}^n\int_M|\nabla h_{t_i}|^2\omega_{t_i}^n\notag \\&\leq \frac{1}{4}\int_M |{\rm Hess}\,h_{t_i}|^2\omega_{t_i}^n+(C^2+1)\int_M|\nabla h_t|^2\omega_{t_i}^n.
  \end{align}
  Combined this with (\ref{gr}), we prove (\ref{ricl}).

 We may assume that $\lim_{i\rightarrow \infty}\lambda_1(t_i)=T<\infty$.  Let  $f_i$ be the corresponding eigenfunction normalized with $\int_M f_i^2=1$.  Then by  Moser's iteration,  we have $|\nabla f_i|\leq C$. By the Bochner formula
   $$\Delta |\nabla f_i|^2={\rm Ric}\,(\omega_{t_i})(\nabla f_i,\nabla f_i)+|{\rm Hess}\,f_i|^2+\langle \nabla f_i,\nabla \Delta f_i\rangle.$$ Integrating both sides, we get
    $$\int_M (1-\lambda_1(t_i))|\nabla f_i|^2\omega_{t_i}^n+\int_M |{\rm Hess}\,f_i|^2\omega_{t_i}^n+\int_M ({\rm Ric}\,(\omega_{t_i})-\omega_{t_i})(\nabla f_i,\nabla f_i)\omega_{t_i}^n=0.$$
     Note that for the third term it holds by (\ref{ricl}),
      $$|\int_M ({\rm Ric}\,(\omega_{t_i})-\omega_{t_i})(\nabla f_i,\nabla f_i)|\omega_{t_i}^n\leq C |\int_M |{\rm Ric}\,(\omega_{t_i}) -\omega_{t_i}|\omega_{t_i}^n\rightarrow 0.$$
       Hence, we obtain  $T\geq 1$.
  \end{proof}

We also need the following lemma.

  \begin{lem}\label{Killing}Let $\Delta$ be the  Laplace operator associated to $\omega_\infty$. Suppose that $u$ satisfies:
   $$\Delta u=-u, |\nabla u|\leq C,~ {\rm on}~ {\rm Reg}( \tilde M_\infty).$$
  Then $Y=\nabla u$ is a Killing VF.
  \end{lem}
  \begin{proof}
   By the Weitzenb\"{o}ch formula, we have
  $$\Delta_{\bar\partial} \bar\partial u\,=\,\bar\nabla^* \bar\nabla \bar\partial u\,+\,{\rm Ric}\,(\bar\partial u,),$$
  where $\Delta_{\bar\partial}$ is the Hodge laplacian.
  Let $\gamma_\epsilon$ be the cut-off function in Lemma 6.10 in \cite{T2}.   Then by multiplying  both sides of by $\gamma_\epsilon^2\bar\partial u$ to the above identity, we have
  $$\int_{\tilde M_\infty}\gamma^2_\epsilon \langle\Delta_{\bar\partial} \bar\partial u,\bar\partial u\rangle\,\omega^n\,=\,
  \int_{\tilde M_\infty}\,\langle\bar\nabla^* \bar\nabla \bar\partial u, \gamma^2_\epsilon\bar\partial u\rangle\,\omega^n\,+\,\int_{\tilde M_\infty}\gamma^2_\epsilon|\bar\partial u|^2\,\omega^n.$$
  Since $\Delta_{\bar\partial} u=u$,  we get
  $$\int_{\tilde M_\infty}\gamma^2_\epsilon \langle\Delta_{\bar\partial} \bar\partial u,\bar\partial u\rangle\,\omega^n\,=\,\int_{M_\infty}\gamma^2_\epsilon \langle\bar\partial \Delta_{\bar\partial} u,\bar\partial u\rangle\,\omega^n\,=\,\int_{M_\infty}\gamma^2_\epsilon|\bar\partial u|^2\,\omega^n.$$
  Thus we derive
   $$\int_{\tilde M_\infty}\,\langle\bar\nabla^* \bar\nabla \bar\partial u, \gamma^2_\epsilon\bar\partial u\rangle\,\omega^n=0.$$
   By integration by parts, it follows that
  $$\int_{\tilde M_\infty}\langle\bar\nabla^* \bar\nabla \bar\partial u, \gamma^2_\epsilon\bar\partial u\rangle\,\omega^n\,=\,\int_{M_\infty}\gamma^2_\epsilon\langle \bar\nabla \bar\partial u,\bar\nabla \bar\partial u\rangle\,\omega^n\,+\,2\int_{M_\infty} \langle \gamma_\epsilon
  \bar\nabla \bar\partial u,\bar\nabla \gamma_\epsilon\otimes\bar
  \partial u\rangle\,\omega^n.$$
  Note that
  $$2|\int_{\tilde M_\infty} \langle \gamma_\epsilon\bar\nabla \bar\partial u,\bar\nabla \gamma_\epsilon\otimes\partial u\rangle\,\omega^n|\,\leq\, \eta\int_{\tilde M_\infty}\gamma^2_\epsilon\langle \bar\nabla \bar\partial u,\bar\nabla \bar\partial u\rangle\,\omega^n+ \frac{C}{\eta}\int_{\tilde M_\infty}|\nabla \gamma_\epsilon|^2\omega^n,\,\forall \eta>0.$$
  Hence, we get
  $$(1-\eta)\int_{\tilde M_\infty}\gamma^2_\epsilon\langle \bar\nabla \bar\partial u,\bar\nabla \bar\partial u\rangle\,\omega^n\,\leq\,\frac{C}{\eta}\int_{M_\infty}|\nabla \gamma_\epsilon|^2\,\omega^n.$$
  Taking $\epsilon\rightarrow 0$, and then $\eta\rightarrow 0$, we obtain
  $$ \bar\nabla \bar\partial u=0,\text{ in }{\rm Reg}(\tilde M_\infty)$$ which means that $\nabla u$ is a Killing vector field.
  \end{proof}

  \begin{proof}[Proof of Proposition \ref{red}] Let $\eta_\infty$ be the Lie algebra of ${\rm Aut}_0(\tilde M_\infty)$.  Then,  as in Lemma 6.9 in \cite{T2}, for  any  holomorphic VF $X\in \eta_\infty$ on $\tilde M_\infty$,  there is a bounded function $\theta_\infty$ satisfying
  \begin{align}\label{theta-eingenvalue}i_X\omega_\infty\,=\,\sqrt{-1}\,\bar\partial \,\theta_\infty,\,\,\,\, \Delta \,\theta_\infty\,=\, -\,T\,\theta_\infty,\,\,
  \text{ in } \tilde M_\infty\setminus \mathcal{S},
  \end{align}
where  $\theta_\infty=u+\sqrt{-1}v$.  We claim:

There is a sequence $\{u_j\}$ of eigenfunctions on $(M,\omega_j)$ such that
$$\Delta \,u_j\,=\,-\,\lambda_j u_j,~\lambda_j\rightarrow 1,$$
 and $u_j$ converges to a Lischitz function $u$ on $M_\infty$ satisfying
  $$\Delta u=-u, \text{ in } {\rm Reg}( \tilde M_\infty).$$

  Denote the set of such above limit eigenfunctions by $\tilde\Lambda_1$ which is a subset of $\Lambda_1$ consisting of all bounded eigenfunctions with eigenvalue $1$.   If  the above claim is not true, namely, $\tilde\Lambda_1\neq \Lambda_1$,  there is a $u\in \Lambda_1 $ such that
  $$\int_{\tilde M_\infty}\,u^2\,\omega_\infty^n\,=\,1,\,\,\,\, \int_{M_\infty}\,u u_a\,\omega_\infty^n\,=\,0,$$
  where $\{u_a\}_{1\leq a\leq k}$ is an orthonormal basis of $\tilde\Lambda_1$.
  Because $\mathcal{S}$ is a subvariety which is contained in a divisor, as in Lemma \ref{Killing}, we have a cut-off function in $\tilde M_\infty$ satisfying
  $$\int_{\tilde M_\infty}\,|\nabla \gamma_\epsilon|^2\,\omega_\infty^n\,\leq\, \epsilon.$$
  On the support $K_\epsilon$ of $\gamma_\epsilon$,  for $\gamma=\gamma(\epsilon)$ we have $\Psi_\gamma^i: K_\epsilon \rightarrow (M,\omega_i)$ such that $(\Psi_\gamma^i)^*\omega_i\rightarrow\omega_\infty$ smoothly. Thus  by taking a subsequence if necessary, we can take $\epsilon_i\rightarrow 0$ such that
  $u_i\,=\,(\gamma_{\epsilon_i}\,u)\circ (\Psi_\gamma^i)^{-1}$ satisfies
  $$\lim_{i\rightarrow\infty}\,\int_M\, |\nabla u_i|^2\,\omega_{i}^n\,=\,1, \,\,\,\,\lim_{i\rightarrow\infty}\,\int_M\, u_i^2\,\omega_{i}^n\,=\,1.$$
 On the other hand,  for each $a$, there are eigenfunctions $u_{a,i}$ of $(M,\omega_j)$ which converge to $u_a$. Then
  $u_i,u_{1,i},...,u_{k,i}$ is a $k+1$ dimensional subspace for large enough  $i$.  Thus  we can find an eigenfunction $u_{0,i}$ orthogonal to $u_{a,i}(1\leq a\leq k)$ with eigenvalue not bigger than $1+\nu_i$ with $\nu_i\rightarrow 0$ by variational principle.
  However, by Lemma \ref{eigenvalue}, we know that the eigenvalue is not less than $1+o(1)$. Hence,  $u_{0,i}$ will converge to an element in $\tilde\Lambda_1$. It's a contradiction! The claim is true.

  By the above claim,  we see that $u$ in (\ref{theta-eingenvalue})  with  the normalization $\int_{\tilde M_\infty} u^2=1$  is a limit of eigenfunctions $u_j$  with $\int_{M} u_j^2\omega_{t_i}^n=1$ on $(M,\omega_j)$.
   By the Moser iteration,  we get
     $$|\nabla u|\leq C$$
     for some $C>0$.  Thus, by Lemma \ref{Killing}, $Y$ is a Killing VF.  Similarly, $\nabla v$ is also a Killing VF. This proves that $\eta_\infty$
 is reductive.

  \end{proof}

 Now we consider  the  soliton  case for $(\tilde M_\infty, \hat\omega_\infty)$.   Denote the subgroup of ${\rm Aut}_0(\tilde M_\infty)$ commuting with $v$ by ${\rm Aut}^v(\tilde M_\infty)$,  where $v$ is  the soliton VF of $(\tilde M_\infty, \hat\omega_\infty)$.  We prove

 \begin{prop}\label{redsol}
 ${\rm Aut}^v(\tilde M_\infty)$ is reductive.
 \end{prop}

 Consider the operator
 $$L(\psi)=\Delta \psi+X(\psi)+\psi, \psi\in C^\infty({\rm Reg}(\tilde M_\infty)), ~\psi\in {\rm W}^{1,2}(\tilde M_\infty).$$
  Then $L$ is an self-adjoint operator  with respect to the following inner product:
 $$\langle f,g\rangle=\int_{\tilde M_\infty} f\bar g e^{h_\infty}{\hat \omega_\infty}^n.$$
 We  want to show that  $|\nabla u|\leq C$ for any $u\in  {\rm ker} (L)$, where $C$ is a uniform constant.

 The following lemma can be found in \cite{Fu, TZ3}.

 \begin{lem}\label{poincare-eigenvalue}
 Let  $L_i(\psi)=\Delta_{\omega_{t_i}}(\psi)+(h_{t_i})_{\bar l}\psi_l+\psi$. Then $\lambda_i\geq 0$, for any non-zero function  $\psi$ satisfying
 $$L_i \psi=-\lambda \psi.$$
 \end{lem}

By the argument in  the proof of Proposition \ref{red}, we prove

 \begin{lem}\label{bg}
 For $u\in  {\rm ker}( L)$, there is a sequence of functions $f_i$ converging to $u$ such that
 \begin{align}\label{condition-approximation}L_i f_i=-\lambda_i f_i,~\lambda_i\rightarrow 0.
 \end{align}
   As a consequence, $|\nabla u|\leq C$.
 \end{lem}

 \begin{proof}
 We denote the subspace of ${\rm ker}( L)$ consisting of $u$ satisfying (\ref{condition-approximation}) by $W$.  If $W$ is a proper subspace of ${\rm Ker}( L)$, we can choose a function $u\in {\rm Ker} (L)$ which is perpendicular to  $W$.  Note that the eigenfunctions of $L_i$ corresponds to the critical points of functional $\int_M |\nabla \psi|^2 e^{h_{t_i}}\omega_{t_i}^n$.  Then   there is  a sequence of eigenfucntions $u_i$ of $L_i$ with eigenvalue not bigger than $o(1)$. By Lemma \ref{poincare-eigenvalue},   one can show that $u_i$ converges to a new function $u'\in {\rm Ker}( L)$ satisfying (\ref{condition-approximation}), namely, $u'\in W$.  However,   $u\in W^{\perp}$  by the construction of $u_i$, which  is a contradiction!

 By using Moser's iteration,   $|\nabla f_i|$ is uniformly bounded for the normalized $f_i$ with
  $\int_M |f_i|^2  e^{h_{t_i}}\omega_{t_i}^n=1$ in (\ref{condition-approximation}). Thus $|\nabla u|\leq C$ for some constant $C$.
 \end{proof}

 \begin{lem}\label{Lint}
 Given $u',v'\in C^\infty(reg(\tilde M_\infty))$. Suppose that
  $u'$ is Lipschitz and $v'$ is bounded uniformly.  Then
  $$\int_{\tilde M_\infty} (\Delta u'+ v(u'))v'e^{h_\infty}\omega_\infty^n=
 -\int_{\tilde M_\infty} u_l' v_{\bar l}'e^{h_\infty}\omega_\infty^n.$$
 \end{lem}

 \begin{proof}
 Choosing the cut-off function $\gamma_\epsilon$ as in Lemma \ref{Killing}, we have
 \begin{align}
 0&=\int_{\tilde M_\infty}(\gamma_\epsilon e^h_\infty v' u_l')_{\bar l} \omega_\infty^n \notag \\
 &=\int_{\tilde M_\infty}\gamma_\epsilon (\Delta u'+ v(u'))v'e^{h_\infty}\omega_\infty^n+
 \int_{\tilde M_\infty}\gamma_\epsilon u_l' v_{\bar l}'e^{h_\infty}\omega_\infty^n+\int_{\tilde M_\infty}({\gamma_\epsilon})_{\bar l}v' u_l'e^h_\infty \omega_\infty^n. \notag
 \end{align}
Note that
 $$|\int_{\tilde M_\infty}({\gamma_\epsilon})_{\bar l}v' u_l'e^h_\infty \omega_\infty^n|\leq  C(\int_{\tilde M_\infty}|\nabla \gamma_\epsilon|^2\omega_\infty^n)^{\frac{1}{2}}.$$
 Letting $\epsilon\rightarrow 0$, the lemma is proved.
 \end{proof}

By Lemma \ref{poincare-eigenvalue}-Lemma \ref{Lint},  we  are able to prove

 \begin{lem}
 $$\eta_\infty\cong {\rm Ker}( L).$$
  The isomorphism is given by the potential of the holomorphic vector field.
 \end{lem}

 \begin{proof}
 By Lemma \ref{bg},  $|\nabla \psi|\leq C$ for any $\psi \in {\rm Ker} (L)$.  Then we can use the integral by parts to get
 $$\int_{\tilde M_\infty} \psi_{kl}\psi_{\bar k\bar l}e^{h_{\infty}}\hat\omega_{\infty}^n=0, ~\forall \psi\in {\rm Ker}( L).
 $$
  Thus $\psi_{kl}=0$ and $\bar \partial \psi$ is the potential of a holomorphic VF.  Conversely for $Z\in \eta_\infty$, write $i_Z\omega_\infty=\bar\partial  \psi$. By Lemma 5.5 in \cite{TW} $\Delta \psi+v(\psi)$ is Lipshitz. Thus,  by Lemma \ref{Lint}, we get
 \begin{align}
 \int_{\tilde M_\infty}L \psi (\Delta \psi+v(\psi))e^h_\infty \hat\omega_\infty^n&=-\int (L\psi)_l\psi_{\bar l}e^h_\infty \hat\omega_\infty^n\notag \\
 &= \int_{\tilde M_\infty} \psi_{kl}\psi_{\bar k\bar l}e^h_\infty \hat\omega_\infty^n=0.
 \end{align}
 Since the eigenvalue of $L$ is nonnegative by Lemma \ref{bg}, we must have $L\psi=0$.
 \end{proof}

 \begin{proof}[Proof of Proposition  \ref{redsol}] As in Appendix in \cite{TZ1}, we
 define $\bar L$ by $\bar L \psi=\overline{L \bar \psi}.$ Let $\Lambda_\lambda$ be the eigenspace of $\bar L$ with eigenvalue $\lambda$. Then we have the follow subspaces  of ${\rm Ker}( L)$:
 \begin{align}
 E_0&={\rm Ker}( L)\bigcap {\rm Ler} (\bar L),\notag \\
 E_0'&=\{f\in  {\rm Ker}( L)\bigcap {\rm Ler} (\bar L)  |~f\text { is real}\},\notag \\
 E_0''&=\{f\in  {\rm Ker}( L)\bigcap {\rm Ler} (\bar L)|~\sqrt{-1}f\text { is real}\},\notag\\
 E_\lambda&={\rm Ker}( L)\bigcap \Lambda_\lambda.\notag
 \end{align}
 By Lemma A.2 in \cite{TZ1}, $[v,Y]=0$ holds if and only if the potential of $Y$ lies in $E_0$. Thus $E_0$ is isomorphic to $\eta^v_\infty$, which is th Lie algebra of $Aut^X(\tilde M_\infty).$  Since  $L_Z \omega_\infty=\sqrt{-1}\partial \bar\partial f$ for  a real function $f=\theta_Z \in E_0'$,   we know that ${\rm Im} Z$ is a Killing vector field.  Hence,  $E_0$ is the complexification of the Lie algbra of Killling  VF. The proposition is proved.
 \end{proof}

  \subsection{Uniqueness of soliton VFs on a $Q$-Fano variety}

  On a $Q$-Fano variety $\tilde M_\infty$, the modified Futaki-invariant in \cite{TZ2} is also well defined (cf. \cite{DT, Xio}),
  $$F_v(X)=\int_{\tilde M_\infty} X(\tilde h_\infty-\theta_v)e^{\theta_v}\tilde \omega_\infty^n=\int_{\tilde M_\infty} X(\tilde h_\infty)e^{\theta_v}\tilde \omega_\infty^n-\int_{\tilde M_\infty}\langle X, v\rangle e^{\theta_v}\tilde \omega_\infty^n,$$
  for any $v, X\in \eta_\infty$.

  By \cite{TZ2}, we prove

  \begin{prop}\label{center}
   Let ${\rm Aut}_r(\tilde M_\infty)\subseteq{\rm Aut}(\tilde M_\infty) $ be a  reductive subgroup with Lie algebra  $\eta_r(\tilde M_\infty)$. Then there is unique holomorphic vector field $v\in \eta_r(\tilde M_\infty)$ such that $F_X$ vanishes. Moreover, $v$ lies in the center of $\eta_r(\tilde M_\infty)$.
  \end{prop}

  \begin{proof}
  Let $ k$ be the Lie algebra of  maximal compact subgroup of ${\rm Aut}_r(\tilde M_\infty)$.   Then  for any $Z\in\eta_r(\tilde M_\infty) $ with  ${\rm Im} Z\in k$,  $\theta_Z$ is a real function. Since
  $$\bar\partial (\Delta_{\tilde \omega_\infty} \theta_Z+Z(\tilde h_\infty)+\theta_Z)=0,  ~{\rm in} ~{\rm reg}(\tilde M_\infty),$$
  $\Delta_{\tilde \omega_\infty} \theta_Z+Z(\tilde h_\infty)+\theta_Z$ is constant.
  Thus we can normalize $\theta_Z$  by
  \begin{align}\label{norm}
 \Delta_{\tilde \omega_\infty} \theta_Z+Z(\tilde h_\infty)+\theta_Z=0.
  \end{align}
 Note that  $\theta_v, \theta_X$ are both  Lipschitz functions.  Hence,  using integral by part as Lemma 5.6 in \cite{TW}, we get
  $$F_v(X)=-\int_{\tilde M_\infty} \theta_X e^{\theta_v} {\tilde \omega_\infty}^n.$$

  Define a function on   $\eta_r(\tilde M_\infty)$ by
  $$f(Z)=\int_{\tilde M_\infty} e^{\theta_Z}\tilde \omega_{\infty}^n.$$
  Then $f$ is a convex function.    We claim that $f$ is a proper function on $\eta_r(\tilde M_\infty)$. Let $E_i(1\leq i\leq m)$ be a basis of $\eta_r(\tilde M_\infty)$ as a  real vector space.

  For any sequence $Z_i\in \eta_r(\tilde M_\infty)$ with
  $\int_{\tilde M_\infty} |Z_i|^2\omega_{\infty}^n\rightarrow \infty$, we have to show that $f(Z_i)\rightarrow \infty$.
  Writing $Z_i=\sum_{j=1}^m a_{i}^jE_j$.  Without loss of generality, we may assume that
  $$|a_i^1|\geq |a_i^j|,(1\leq j\leq m), |a_i^1|\rightarrow \infty$$
   after a subsequence.  We can also assume that $a_i^1>0$ by changing $E_1$ to $-E_1$. By taking a  subsequence again, $\frac{a_i^j}{a_i^1}$ converges as $i\rightarrow \infty$ for $1\leq j\leq m$.  It follows that
     $$E_1+\sum_{j=2}^m \frac{a_i^j}{a_i^1} E_j \rightarrow E\in \eta_r(\tilde M_\infty) (i\rightarrow \infty).$$
Choose an open set $U\subset {\rm reg}(\tilde M_\infty)$ such that $\theta_{E}\geq 2\epsilon$ for some $\epsilon>0$ and then $\sum_{j=1}^m \frac{a_{i}^j}{a_i^1}\theta_{E_j}\geq \epsilon $ for $i$ large enough. Thus
  \begin{align}
  f(Z_i)&=\int_{\tilde M_\infty} e^{\sum_{j=1}^m a_{i}^j\theta_{E_j}}=\int_{\tilde M_\infty} e^{a_i^1} e^{\sum_{j=1}^m \frac{a_{i}^j}{a_i^1}\theta_{E_j}}\notag \\
  &\geq \int_U e^{\epsilon a_i^1}\rightarrow \infty.
  \end{align}
  This proves the claim.
  As a consequence,  $f$ has a unique critical point $v: df_v=0$ and this is equivalent to $F_v(\cdot)\equiv 0.$

 By  restricting  the function $f$ on  the center $\eta_c$ of $\eta_r(\tilde M_\infty)$, there is a unique VF $v'$ such that $F_{v'}(\cdot)\equiv0$ on $\eta_c$. On the other hand,
  for $X,Y\in \eta_r(\tilde M_\infty)$,  a direct computation shows that
  $$F_{v'}(Ad_Y X)=F_{v'}(X).$$
   It follows that
   $$F_{v'}([X,Y])=0.$$
    Since $\eta_r(\tilde M_\infty)/\eta_c$ is semi-simple, we have
    $$F_{v'}(X)\equiv 0, ~\forall  X\in\eta_r(\tilde M_\infty).$$
    Thus $v=v'$ and $v$ lies in the center of $\eta_r(\tilde M_\infty)$.
  \end{proof}

  \end{document}